\documentclass[a4paper, 11pt, twoside]{scrartcl}
\pdfoutput=1

\title{On Strong Stability of Explicit Runge--Kutta Methods for Nonlinear
       Semibounded Operators}
\author{Hendrik Ranocha}
\date{21 May 2019}

\usepackage[utf8]{luainputenc}
\usepackage[UKenglish]{babel}
\usepackage{csquotes}

\usepackage[a4paper, top=0.51in, bottom=0.79in, left=0.999in, right=0.99in]{geometry}

\usepackage[plainpages=false,pdfpagelabels,hidelinks,unicode]{hyperref}
\makeatletter
\hypersetup{pdfauthor={\@author}}
\hypersetup{pdftitle={\@title}}
\makeatother

\usepackage[%
  backend=biber,
  style=numeric-comp,
  giveninits=true, uniquename=init, 
  natbib=true,
  url=true,
  doi=true,
  isbn=false,
  backref=false,
  maxnames=99,
  ]{biblatex}
\usepackage{amsmath}
\allowdisplaybreaks
\numberwithin{equation}{section}
\usepackage{amssymb}
\usepackage{commath}
\usepackage{mathtools}
\usepackage{bbm}
\usepackage{nicefrac}

\usepackage{siunitx}

\usepackage{amsthm}
\usepackage{thmtools}
\usepackage{etoolbox}
\makeatletter
\patchcmd{\thmt@setheadstyle}
 {\bgroup\thmt@space}
 {\thmt@space}
 {}{}
\patchcmd{\thmt@setheadstyle}
 {\egroup\fi}
 {\fi}
 {}{}
\makeatother
\declaretheoremstyle[
  bodyfont=\normalfont\itshape,
  headformat=\NAME\ \NUMBER\NOTE,
]{myplain}
\declaretheoremstyle[
  headformat=\NAME\ \NUMBER\NOTE,
]{mydefinition}
\newcommand{\envqed}{{\lower-0.3ex\hbox{$\triangleleft$}}}
\declaretheorem[style=myplain,numberwithin=section]{theorem}
\declaretheorem[style=myplain,numberwithin=section]{lemma}

\declaretheorem[style=mydefinition,numberwithin=section,qed=\envqed]{definition}
\declaretheorem[style=mydefinition,numberwithin=section,qed=\envqed]{remark}

\usepackage{ifluatex}
\ifluatex
  \usepackage[no-math]{fontspec}
\else
  \usepackage[T1]{fontenc}
\fi
\usepackage{newpxtext,newpxmath}

\usepackage{color}
\usepackage{graphicx}
\usepackage[small]{caption}
\usepackage{subcaption}

\usepackage{pgfplots}
\usepackage{tikzscale}
\pgfplotsset{compat=1.13}
\usepgfplotslibrary{groupplots}
\usepgfplotslibrary{external}
\begingroup\expandafter\expandafter\expandafter\endgroup
\expandafter\ifx\csname pdfsuppresswarningpagegroup\endcsname\relax
\else
\fi

\usepackage{booktabs}
\usepackage{rotating}
\usepackage{multirow}

\makeatletter
\newcommand\mynewtag[2]{#1\def\@currentlabel{#1}\label{#2}}
\makeatother

\usepackage{multicol}
\usepackage{enumitem}

\usepackage{calc}
\usepackage{xparse}

\renewcommand{\vec}[1]{\underline{#1}}
\NewDocumentCommand{\mat}{mo}{%
  \IfValueTF{#2}{%
    \underline{\underline{#1}}{#2}
  }{%
    \underline{\underline{#1}}\,
  }%
}
\newcommand{\diag}[1]{\operatorname{diag}\left(#1\right)}

\AtBeginDocument{%
  
}

\newcommand{\scp}[2]{\left\langle{#1,\, #2}\right\rangle}

\DeclarePairedDelimiterX\newset[1]\lbrace\rbrace{\setaux #1||\endsetaux}
\def\setaux#1|#2|#3\endsetaux{\if\relax\detokenize{#2}\relax #1 \else #1 \;\delimsize\vert\; #2 \fi}
\renewcommand{\set}[1]{\newset*{#1}}

\newcommand{\vect}[1]{\begin{pmatrix} #1 \end{pmatrix}}

\newcommand{\I}{\operatorname{I}}
\newcommand{\fnum}{f^{\mathrm{num}}}

\newcommand{\vecfnum}{\vec{f}^{\mathrm{num}}}

\newcommand{\Fnum}{F^{\mathrm{num}}}
\newcommand{\vecFnum}{\vec{F}^{\mathrm{num}}}

\AtBeginDocument{%
  \let\rho\varrho
  \let\phi\varphi
  \let\epsilon\varepsilon
}
\newcommand{\N}{\mathbb{N}}
\newcommand{\R}{\mathbb{R}}

\newcommand{\dt}{\Delta t}
\renewcommand{\H}{\mathcal{H}}
\renewcommand{\L}{\mathcal{L}}
\newcommand{\dtFE}{\dt_{E}}
\renewcommand{\O}{\mathcal{O}}
\renewcommand{\i}{\mathrm{i}}

\newcommand{\unum}{u_\mathrm{num}}
\newcommand{\normLip}[1]{\abs{#1}_{\mathrm{Lip}}}

\newcommand{\sep}{$\cdot\,$}

\begin{document}

\maketitle

\begin{abstract}
  Explicit Runge--Kutta methods are classical and widespread techniques in the numerical
solution of ordinary differential equations (ODEs). Considering partial differential
equations, spatial semidiscretisations can be used to obtain systems of ODEs that are
solved subsequently, resulting in fully discrete schemes. However, certain stability
investigations of high-order methods for hyperbolic conservation laws are often conducted
only for the semidiscrete versions.
Here, strong stability (also known as monotonicity) of explicit Runge--Kutta methods
for ODEs with nonlinear and semibounded (also known as dissipative) operators is
investigated. Contrary to the linear case, it is proven that many strong stability
preserving (SSP) schemes of order two or greater are not strongly stable for general
smooth and semibounded nonlinear operators. Additionally, it is shown that there
are first order accurate explicit SSP Runge--Kutta methods that are strongly stable
(monotone) for semibounded (dissipative) and Lipschitz continuous operators.


\end{abstract}

\section{Introduction}
\label{sec:introduction}

Considering the numerical solution of (partial) differential equations, stability
of the schemes plays an important role. For linear symmetric hyperbolic partial
differential equations (PDEs), energy estimates can often be obtained, resulting
in both uniqueness of solutions and existence via appropriate approximations,
as described in the monographs of \citet{kreiss2004initial, gustafsson2013time}.
Using schemes in the framework
of summation by parts operators \citep{kreiss1974finite} with simultaneous
approximation terms \citep{carpenter1994time, carpenter1999stable}, these energy
estimates can often be transferred to semidiscrete schemes, as described in the
review articles of \citet{svard2014review, fernandez2014review} and references cited
therein. While this framework has been developed originally for finite difference
schemes, it contains also many other classes of methods such as finite volume
\citep{nordstrom2001finite, nordstrom2003finite}, discontinuous Galerkin
\citep{gassner2013skew, fernandez2014generalized}, and flux reconstruction schemes
\citep{huynh2007flux, ranocha2016summation}.

Since numerical methods are used to obtain fully discrete schemes from semidiscrete
ones, the preservation of such kind of semidiscrete stability is worth investigating.
Strong stability preserving (SSP) methods can be written as convex combinations of
explicit Euler steps. Hence, they preserve all convex stability properties of the
explicit Euler method, as described in the monograph of \citet{gottlieb2011strong}
and references cited therein.

However, even for linear ODEs with semibounded operators, the explicit Euler method
does not preserve $L^2$ stability in general. Thus, the SSP property cannot be used
to obtain strong stability for these schemes. Nevertheless, some well-known
high order, explicit SSP Runge--Kutta methods are strongly stable in this general
case \citep{tadmor2002semidiscrete, ranocha2018L2stability}. While the classical
fourth order Runge--Kutta method is not strongly stable after one time step, combining
two consecutive time steps results in a strongly stable scheme for this class of
problems \citep{sun2017stability}. Further results for linear autonomous systems
have been obtained by \citet{sun2018strong}.

Because many hyperbolic conservation laws are nonlinear or semidiscretisations
are obtained via nonlinear processes, it is interesting whether explicit SSP
Runge--Kutta methods can be strongly stable for general nonlinear ODEs with
semibounded operators. In the context of hyperbolic conservation laws, many
studies are based on the seminal work of \citet{tadmor1987numerical, tadmor2003entropy}
concerning entropy stability of semidiscretisations. While there
are related studies of explicit Runge--Kutta methods \citep{lozano2018entropy},
there are no general results on strong stability.

Although strong stability can be considered for general convex functionals, the
$L^2$ norm will be used in this article. It is most similar to the linear case and
of interest in applications, e.g. in the recent article of \citet{nordstrom2018energy}.
Moreover, it is a
special case of strong stability and it might be expected that there is a larger
set of methods that are strongly stable for this convex functional, similar to
the case of SSP methods studied by \citet{higueras2005monotonicity}. Furthermore,
the focus is on explicit schemes, since they are widespread, can be implemented
easily, are often more efficient if accuracy is a determining factor, and efficient
use of parallelism is less expensive than for implicit schemes \citep{kopriva2013assessment}.

In this article, it is proven that many explicit SSP Runge--Kutta methods of order
two or greater cannot be strongly stable for nonlinear ODEs with smooth and semibounded
operators in general. Some tedious calculations used in these proofs are
verified using Mathematica~\citep{mathematica10} and published online
\citep{ranocha2018strongRepository}. Moreover, it is shown that first order accurate
schemes can be both SSP and strongly stable for semibounded and Lipschitz continuous
operators.

To do so, the article is structured as follows. At first, basic definitions such
as strong stability and semiboundedness are given in section~\ref{sec:rk-methods}.
Additionally, a brief review of Runge--Kutta methods is included to introduce the
notation. Afterwards, Runge--Kutta methods of up to three stages are studied in
section~\ref{sec:three-stages}. It is shown that there are no such schemes with
order of accuracy of at least two that are strongly stable and SSP. This result
is based on the explicit construction of ODEs with nonlinear and semibounded
operators and implications of the order conditions.

Since the number of parameters and the complexity of the order conditions increases
with the number of stages, the general approach is not really feasible for methods
with more stages. Therefore, some known explicit SSP methods with more than three
stages are investigated separately in
section~\ref{sec:ssprks2-ssprks3}. In particular, it is shown that the families
of schemes with optimal SSP coefficient of order two \cite[Theorem~9.3]{kraaijevanger1991contractivity} and three \cite[Theorem~3]{ketcheson2008highly}
and the ten stage, fourth order method SSPRK(10,4) of \citet{ketcheson2008highly}
are not strongly stable in general.

Thereafter, two well-known and popular SSP methods are studied in more detail
in the following sections.
While the investigations up to this point are only concerned with strong stability
and not with boundedness in general, the popular three-stage method SSPRK(3,3) of
\citet{shu1988efficient} is studied in detail in section~\ref{sec:ssprk33}.
It is shown that the norm of the numerical approximation can increase monotonically
and without bounds for nonlinear and semibounded operators.
Since this article is motivated by applications of SSP methods to semidiscretisations
of hyperbolic conservation laws, an energy stable and nonlinear semidiscretisation
of the linear transport equation is constructed in section~\ref{sec:ssprk104}.
This ODE with semibounded operator is solved numerically with SSPRK(10,4) and it
is shown that the norm of the numerical solution increases for a large range of
time steps.

Turning to first order schemes in section~\ref{sec:first-order}, it is shown that
the limitations of high order schemes studied before do not apply. In particular,
there are explicit SSP Runge--Kutta methods of first order of accuracy that are
strongly stable for semibounded and Lipschitz continuous operators with Lipschitz
constant $L$ under a time step restriction $\dt \leq \dt_\mathrm{max}$, where
$\dt_\mathrm{max} \propto L^{-1}$.
Finally, the results are summarised and discussed in section~\ref{sec:summary}.

\section{Brief Review of Runge Kutta Methods}
\label{sec:rk-methods}

Consider an ordinary differential equation
\begin{equation}
\label{eq:ode}
\begin{aligned}
  \od{}{t} u(t) &= g\bigl( u(t) \bigr),
  && t \in (0,T),
  \\
  u(0) &= u_0,
\end{aligned}
\end{equation}
in a real vector space $\H$ with semi inner product $\scp{\cdot}{\cdot}$, inducing
the seminorm $\norm{\cdot}$. Typically, $\H$ can be a Hilbert space. Therefore,
$\norm{\cdot}$ will be called norm in the following. However, the property distinguishing
a norm from a seminorm will not be used anywhere.

\subsection{Strong Stability}

For a smooth solution of \eqref{eq:ode}, the time derivative of the squared norm is
\begin{equation}
  \od{}{t} \norm{u(t)}^2
  =
  2 \scp{u(t)}{\od{}{t} u(t)}
  =
  2 \scp{u(t)}{g\bigl(u(t)\bigr)}.
\end{equation}

\begin{definition}
  A function $g\colon \H \to \H$ is \emph{semibounded}, if
  \begin{equation}
    \forall u \in \H\colon
    \quad
    \scp{u}{g(u)} \leq 0.
    \qedhere
  \end{equation}
\end{definition}
\begin{remark}
  If a complex (semi) inner product space is considered instead of a real one,
  the real part of the (semi) inner product $\scp{u}{g(u)}$ has to be non-positive.
\end{remark}
\begin{remark}
  Sometimes, such operators $g$ are also called (energy) \emph{dissipative}. Here,
  the term \emph{semibounded} is used instead, in order to emphasise that (energy)
  conservative operators are included in this definition.
\end{remark}

Thus, the (squared) norm of a smooth solution $u$ of \eqref{eq:ode} is bounded by
its initial value if $g$ is semibounded. However, an approximate solution obtained
by a numerical method does not necessarily satisfy this inequality. For example,
applying one step of the explicit Euler method to \eqref{eq:ode} yields the new
value $u_+ = u_0 + \dt g(u_0)$, satisfying
\begin{equation}
  \norm{u_+}^2
  =
  \norm{u_0 + \dt g(u_0)}^2
  =
  \norm{u_0}^2
  + \underbrace{2 \dt \scp{u_0}{g(u_0)}}_{\leq 0}
  + \underbrace{\dt^2 \norm{g(u_0)}^2}_{\geq 0}.
\end{equation}
Thus, for a general semibounded $g$, the norm of the numerical solution can increase
during one time step, e.g. if $\scp{u_0}{g(u_0)} = 0$. In particular, this happens
if $g(u) = \L u$ where $\L$ is a linear and skew-symmetric operator.

\begin{definition}
  A numerical scheme approximating \eqref{eq:ode} during one time step from
  $u_0 \approx u(t)$ to $u_+ \approx u(t + \dt)$ with semibounded $g$ is called
  \emph{strongly stable} if $\norm{u_+}^2 \leq \norm{u_0}^2$.
\end{definition}
\begin{remark}
  Since this work is motivated by discretisations of PDEs, the term strong stability
  is used. In the literature on Runge--Kutta methods, such a property is often
  called \emph{monotonicity}.
\end{remark}

Nevertheless, the explicit Euler method can be strongly stable under stronger
assumptions on $g$. For example, consider the condition
\begin{equation}
\label{eq:g-nonlinear-coercive}
  \exists M \in \R \, \forall u \in \H\colon
  \quad
  \scp{u}{g(u)} \leq M \norm{g(u)}^2.
\end{equation}
If $M < 0$ in \eqref{eq:g-nonlinear-coercive}, the explicit Euler method
$u_+ = u_0 + \dt g(u_0)$ is strongly stable under the time step restriction
$\dt \in (0, -2 M]$, since
\begin{equation}
  \norm{u_0 + \dt g(u_0)}^2 - \norm{u_0}^2
  =
  2 \dt \scp{u_0}{g(u_0)} + \dt^2 \norm{g(u_0)}^2
  \leq
  (\dt + 2 M) \dt \norm{g(u_0)}^2
  \leq
  0.
\end{equation}
Such right hand sides with linear $g$ are called \emph{coercive} by
\citet{levy1998semidiscrete, tadmor2002semidiscrete}.

\begin{remark}
\label{rem:convex-functionals}
  Instead of the norm $\norm{\cdot}$ of the solution, other convex functionals
  can be considered. For semidiscretisations of hyperbolic conservation laws,
  some important examples are the $L^1$ norm $\norm{u(t)}_1 = \int \abs{u(t,x)} \dif x$,
  the total variation seminorm $\norm{u(t)}_{TV}$, non-negativity (expressed via
  $-\min_x u(t,x)$), or the total entropy $\int U\bigl(u(t,x) \bigr) \dif x$,
  where $U$ is a convex function.
  In that case, strong stability refers to the monotonicity of that particular
  convex functional in time.
\end{remark}

Since the explicit Euler method is relatively simple, it is desirable to transfer
results such as strong stability from that method (which are relatively easy to
check) to high order schemes (for which it is considerably more difficult to check
these properties). Strong stability preserving numerical schemes are designed to
enable exactly this transfer, as described in the monograph of \citet{gottlieb2011strong}
and references cited therein.
\begin{definition}
  A numerical time scheme for \eqref{eq:ode} is called \emph{strong
  stability preserving} (SSP) with SSP coefficient $c > 0$ if it is strongly
  stable under the time step restriction $\dt \leq c \dtFE$ whenever the explicit
  Euler method is strongly stable for $\dt \leq \dtFE$ and any convex functional,
  i.e. if for all convex functionals $\eta$,
  $\forall \dt \in (0,\dt_E]\colon \eta(u_0 + \dt \, g(u_0)) \leq \eta(u_0)$
  implies
  $\forall \dt \in (0,c \dt_E]\colon \eta(u_+) \leq \eta(u_0)$.
\end{definition}
Considering only semi inner products as in this article, some restrictions for
general SSP methods can be relaxed \citep{higueras2005monotonicity}.

\subsection{Connections to Other Properties}

The application of a one-sided Lipschitz condition
\begin{equation}
\label{eq:f-logarithmic-Lipschitz}
  \exists M \in \R\, \forall t \in [0,T], u,v \in X\colon
  \quad
  \scp{f(t,u) - f(t,v)}{u-v} \leq M \norm{u-v}^2
\end{equation}
has been very successful, e.g. for stiff and dissipative problems
\begin{equation}
\label{eq:ode-f}
\begin{aligned}
  \od{}{t} u(t) &= f\bigl( t, u(t) \bigr),
  && t \in (0,T),
  \\
  u(0) &= u_0,
\end{aligned}
\end{equation}
cf. \citet{soderlind2006logarithmic}. Indeed, the existence of such a one-sided/logarithmic
Lipschitz constant $M$ implies that the difference between two solutions $u,v$
of \eqref{eq:ode-f} with initial conditions $u_0,v_0$ remains bounded. In particular,
if $M \leq 0$, the difference between two solutions does not increase, resulting
in \emph{contractivity}. Thus, such a condition yields some important stability/boundedness/robustness
properties. Since \eqref{eq:f-logarithmic-Lipschitz} does not restrict the Lipschitz
seminorm
\begin{equation}
  \normLip{f} := \sup_{u \neq v} \frac{\norm{f(t,u) - f(t,v)}}{\norm{u-v}}
\end{equation}
of $f$, results based thereon can be applied to arbitrarily stiff equations.
Hence, it is mostly useful for implicit methods.
In order to be able to investigate also stability properties of explicit methods,
\citet{dahlquist1979generalized} introduced the condition
\begin{equation}
\label{eq:f-contractive}
  \exists M \in \R\, \forall t \in [0,T], u,v \in X\colon
  \quad
  \scp{f(t,u) - f(t,v)}{u-v} \leq M \norm{f(t,u) - f(t,v)}^2,
\end{equation}
see also \citet[Chapter~6]{dekker1984stability}. If $f$ satisfies \eqref{eq:f-contractive}
with $M < 0$, $f$ is Lipschitz continuous in its second argument with
$\normLip{f} \leq -M^{-1}$, since
\begin{equation}
  -M \norm{f(t,u) - f(t,v)}^2
  \leq
  - \scp{f(t,u) - f(t,v)}{u-v}
  \leq
  \norm{f(t,u) - f(t,v)} \, \norm{u - v}.
\end{equation}
$M < 0$ yields again a contractive ODE and results based on \eqref{eq:f-contractive}
can be applied to explicit methods since the Lipschitz constant of $f$ is bounded.

General results on contractivity can be implied by seemingly simpler requirements
such as
\begin{equation}
\label{eq:f-logarithmic-monotone}
  \exists M \in \R \, \forall t \in [0,T], \tilde u \in \tilde X\colon
  \quad
  \scp{\tilde f(t, \tilde u)}{\tilde u} \leq M \norm{\tilde u}^2
\end{equation}
or
\begin{equation}
\label{eq:f-monotone}
  \exists M \in \R \, \forall t \in [0,T], \tilde u \in \tilde X\colon
  \quad
  \scp{\tilde f(t, \tilde u)}{\tilde u} \leq M \norm{\tilde f(t, \tilde u)}^2,
\end{equation}
cf. \cite{burrage1980nonlinear} or \cite[Section~357]{butcher2016numerical}.
In particular, monotonicity/semiboundedness results such as discrete versions of
$\norm{u(t)} \leq \norm{u_0}$ can be transferred directly to contractivity.
Therefore, only the former will be studied in this article.

\begin{remark}
  For linear ODEs with possibly time dependent coefficients, the concepts of
  contractivity and monotonicity are equivalent. Since many results have been
  established for contractivity, e.g. by \citet{dahlquist1979generalized,
  dekker1984stability}, they can be transferred
  directly to monotonicity. In particular, severe limitations of numerical methods
  result from linear ODEs with varying coefficients.
\end{remark}

\begin{remark}
  Results for right hand sides $f$ satisfying \eqref{eq:f-monotone} with $M < 0$
  have been obtained by \citet{higueras2005monotonicity}, similar to the results
  for circle contractivity by \citet{dahlquist1979generalized, dekker1984stability}.
  This is a special case of strong stability preservation and is more directly
  related to semibounded operators considered in this article. Nevertheless,
  since numerical methods for hyperbolic conservation laws motivate this study,
  general SSP methods will be considered.
\end{remark}

\begin{remark}
  For linear and time-independent ODEs \eqref{eq:ode} with semibounded $g$, some
  strong stability properties have been obtained by \citet{ranocha2018L2stability,
  sun2018strong}.
  Thus, it is interesting whether similar results can be established under the
  assumption \eqref{eq:f-monotone} with $M \leq 0$.
  In order to restrict the stiffness of the ODE \eqref{eq:ode}, a Lipschitz
  condition will be assumed, i.e. $\normLip{g} \leq L$. Since there are many
  negative results even for autonomous problems, \eqref{eq:ode} will be considered
  instead of \eqref{eq:ode-f}.
\end{remark}

\subsection{Runge--Kutta Methods}

A general (explicit or implicit) Runge--Kutta method with $s$ stages can be described
by its Butcher tableau \citep{hairer2008solving, butcher2016numerical}
\begin{equation}
\label{eq:butcher}
\begin{array}{c | c}
  c & A
  \\ \hline
    & b
\end{array}
\end{equation}
where $A \in \R^{s \times s}$ and $b, c \in \R^s$. Since \eqref{eq:ode} is an
autonomous ODE, there is no explicit dependency on time and one step from $u_0$
to $u_+$ is given by
\begin{equation}
\label{eq:RK-step}
  u_i
  =
  u_0 + \dt \sum_{j=1}^{s} a_{ij} \, g(u_j),
  \qquad
  u_+
  =
  u_0 + \dt \sum_{i=1}^{s} b_{i} \, g(u_i).
\end{equation}
Here, $u_i$ are the stage values of the Runge--Kutta method. It is also
possible to express the method via the slopes $k_i = g(u_i)$
\citep[Definition II.1.1]{hairer2006geometric}.

Using the stage values $u_i$ as in \eqref{eq:RK-step},
the change of squared norm (``energy'') is given by
\begin{align*}
\stepcounter{equation}\tag{\theequation}
\label{eq:estimate-RK}
  &
  \norm{ u_+ }^2 - \norm{ u_0 }^2
  =
  2 \dt \scp{ u_0 }{ \sum_{i=1}^{s} b_{i} \, g(u_i) }
  + (\dt)^2 \norm{ \sum_{i=1}^{s} b_{i} \, g(u_i) }^2
  \\\stackrel{\eqref{eq:RK-step}}{=}&\,
  2 \dt \sum_{i=1}^{s} b_{i}
    \scp{ u_i - \dt \sum_{j=1}^{s} a_{ij} \, g(u_j) }{
    g(u_i) }
  + (\dt)^2 \norm{ \sum_{i=1}^{s} b_{i} \, g(u_i) }^2
  \\=&\,
  2 \dt \sum_{i=1}^{s} b_{i} \scp{ u_i }{ g(u_i) }
  + (\dt)^2 \left[
    \norm{ \sum_{i=1}^{s} b_{i} \, g(u_i) }^2
    - 2 \sum_{i,j=1}^{s} b_{i} \, a_{ij}
      \scp{ g(u_i) }{ g(u_j) }
  \right]
  \\=&\,
  2 \dt \sum_{i=1}^{s} b_{i} \scp{ u_i }{ g(u_i) }
  + (\dt)^2 \left[
    \sum_{i,j=1}^{s} \left( b_i b_j - b_{i} \, a_{ij} - b_j a_{ji} \right)
      \scp{ g(u_i) }{ g(u_j) }
  \right],
\end{align*}
where the symmetry of the scalar product has been used in the last step.
Here, the first term on the right hand side is consistent with
$\int_{t_0}^{t_0 + \dt} 2 \scp{u(t)}{g\bigl(u(t)\bigr)} \dif t$,
if the Runge--Kutta method is consistent, i.e. $\sum_{i=1}^{s} b_i = 1$.
Additionally, it can be estimated via the semiboundedness of $g$ if all $b_i$
are non-negative.

The second term of order $(\dt)^2$ is undesired. Depending on the method
(and the stages, of course), it may be positive or negative. However, if it is
positive, then a stability error may be introduced.

As a special case, if the method fulfils
$b_i b_j = b_i a_{ij} + b_j a_{ji},\, i,j \in \set{1, \dots, s}$,
this term vanishes. These methods can conserve quadratic invariants of ordinary
differential equations, a topic of geometric numerical integration, see
Theorem IV.2.2 of \citet{hairer2006geometric}, originally proved by
\citet{cooper1987stability}. A special kind of these methods are the implicit
Gauß methods \citep[Section II.1.3]{hairer2006geometric}.

More generally, the $(\dt)^2$ term is non-positive if the matrix with entries
$(b_i b_j - b_{i} \, a_{ij} - b_j a_{ji})_{i,j}$ is negative semidefinite (and
$b_i \geq 0$, as before), i.e. when the Runge--Kutta method is \emph{algebraically
stable}. Then, the Runge--Kutta method is strongly stable in the $L^2$ norm for
every time step $\dt > 0$, i.e. $B$ stable, cf. \cite[section~357]{butcher2016numerical}.
While there are Runge--Kutta methods with these nice stability properties, these
are all implicit.

\begin{remark}
  Applying explicit methods to \eqref{eq:ode}, it can be expected that time step
  restrictions for strong stability depend on boundedness or Lipschitz constants
  of $g$, e.g. $\dt \leq \dt_\mathrm{max} \propto L^{-1}$. Hence, such restrictions
  on $g$ will be used in the following.
\end{remark}

The following result will be used in the next sections, cf.
\cite[Observation~5.2]{gottlieb2011strong} or \cite{kraaijevanger1991contractivity}.
\begin{lemma}
\label{lem:A-b-geq-0}
  Any Runge--Kutta method with positive SSP coefficient $c > 0$ has non-negative
  coefficients and weights, i.e. $\forall i,j\colon a_{ij} \geq 0, b_i \geq 0$.
\end{lemma}

That the coefficients $a_{ij}, b_i$ of the schemes are non-negative can also
be obtained under other conditions focusing on circle contractivity, cf.
\cite{higueras2005monotonicity}. This implies certain restrictions on the possible
order of the schemes, cf. \cite{kraaijevanger1991contractivity} or
\cite[Section~5.1]{gottlieb2011strong}.

\section{Explicit Methods with Three Stages}
\label{sec:three-stages}

In this section, explicit Runge--Kutta methods with three stages are considered.
Thus, the corresponding coefficients are
\begin{equation}
\label{eq:rk33}
  a_{21}, \;
  a_{31}, \;
  a_{32}, \;
  b_1, \;
  b_2, \;
  b_3.
\end{equation}
Since the proofs of the negative results obtained in this section are easier
if fewer coefficients are considered, third order methods will be investigated
at first. Thereafter, schemes of at least second order of accuracy are studied.

The usual conditions for second order accurate Runge--Kutta methods used
later are $\sum_{j=1}^s b_j = 1$ and $\sum_{j,k=1}^s b_j a_{jk} = \nicefrac{1}{2}$.
Third order schemes have to fulfil $\sum_{j,k,l=1}^s b_j a_{jk} a_{jl} = \nicefrac{1}{3}$
and $\sum_{j,k,l=1}^s b_j a_{jk} a_{kl} = \nicefrac{1}{6}$ additionally
\cite[Section~II.2]{hairer2008solving}.

The basic approach to get negative results can be described as follows. Certain
test problems \eqref{eq:ode} using specific semibounded $g$ are constructed such that
the norm of the numerical solution increases during the first time step for each
$\dt \in (0, \dt_\mathrm{max}]$. Then, this result can be transferred to semibounded
$g$ with $\normLip{g} \leq L$ by considering suitable modifications outside
of a bounded region using Kirszbraun's theorem \cite[Theorem~1.31]{schwartz1969nonlinear}:
\begin{theorem}[Kirszbraun]
  Suppose $S$ is a subset of the Hilbert space $\H$ and $g\colon S \to H$ is Lipschitz
  continuous. Then, $g$ can be extended to all of $\H$ in such a way that the extension
  satisfies the same Lipschitz condition.
\end{theorem}

\subsection{Third Order Methods}

\begin{theorem}
\label{thm:ERK-SSP-3-3}
  There is no explicit Runge--Kutta method that
  \begin{itemize}
    \item is strong stability preserving with positive SSP coefficient,
    \item is of third order of accuracy \& has at most three stages,
    \item and is strongly stable for \eqref{eq:ode} for all smooth and semibounded
          $g$ with $\normLip{g} \leq L$.
  \end{itemize}
\end{theorem}

To prove Theorem~\ref{thm:ERK-SSP-3-3}, the initial value problem \eqref{eq:ode} with
\begin{equation}
\label{eq:ode-u1-r-u2}
\begin{aligned}
  u(t) = \vect{u_1(t) \\ u_2(t)},
  \quad
  g(u) = \alpha (u_1 - r u_2) \vect{-u_2 \\ u_1},
  \quad
  u_0 = \vect{1 \\ 0},
\end{aligned}
\end{equation}
will be used, where $r$ is a real parameter, $\alpha > 0$, and $u_1,u_2$ are
real valued functions.
Since $g$ is given by polynomials in $u$, the squared norm after one step can be
calculated explicitly and is a polynomial in the time step $\dt$.
\begin{lemma}
\label{lem:erk33-1}
  Applying an explicit third order Runge--Kutta method with three stages given by
  the parameters \eqref{eq:rk33} to the ODE \eqref{eq:ode} with \eqref{eq:ode-u1-r-u2}
  yields $\norm{u_+}^2 - \norm{u_0}^2 = \dt^4 p(\dt)$, where $p(\dt)$ is a polynomial
  of the form
  \begin{equation}
    p(\dt)
    =
    \frac{\alpha^4}{12} \left(
      - 5 + 7 r^2 + (a_{31} + a_{32}) (4 - 8 r^2)
    \right)
    + \O(\dt).
  \end{equation}
\end{lemma}
Lemma~\ref{lem:erk33-1} can be proved by direct but tedious calculations and has
been verified using Mathematica~\citep{mathematica10}.

\begin{lemma}
\label{lem:erk33-2}
  An explicit third order Runge--Kutta method with three stages given by the
  parameters \eqref{eq:rk33} that is strongly stable for \eqref{eq:ode} for
  all smooth and semibounded $g$ with $\normLip{g} \leq L$ satisfies
  \begin{equation}
    \frac{7}{8} \leq a_{31} + a_{32} \leq \frac{5}{4}.
  \end{equation}
\end{lemma}
\begin{proof}
  In order to be strongly stable for the ODE \eqref{eq:ode} with \eqref{eq:ode-u1-r-u2},
  the coefficient of the constant term of the polynomial $p(\dt)$ given in
  Lemma~\ref{lem:erk33-1} has to be non-positive, i.e.
  \begin{equation*}
    \frac{12}{\alpha^4} \, p(0) =  - 5 + 7 r^2 + (a_{31} + a_{32}) (4 - 8 r^2) \leq 0.
  \end{equation*}
  This can be reformulated as
  \begin{equation}
  \begin{aligned}
    a_{31} + a_{32} &\leq \frac{5 - 7 r^2}{4 - 8 r^2},
      \quad&\text{if } r^2 &< \frac{1}{2}
    \\
    a_{31} + a_{32} &\geq \frac{5 - 7 r^2}{4 - 8 r^2},
      \quad&\text{if } r^2 &> \frac{1}{2}.
  \end{aligned}
  \end{equation}
  Basically, the assertion is proved by letting $r \to 0$ in the first
  inequality and $r \to \infty$ in the second one. To satisfy the upper bound on
  the Lipschitz constant, $\alpha \to 0$ can be coupled with the limiting process
  on $r$: In this way, the local Lipschitz constant of $g$ \eqref{eq:ode-u1-r-u2}
  around $u_0$ can be made arbitrarily small without changing the results of
  Lemma~\ref{lem:erk33-2}.
  Since only one time step is considered, $g$ can be modified outside of a suitable
  neighbourhood of $u_0$ while keeping the local Lipschitz constant of $g$ as
  global Lipschitz constant because of Kirszbraun's theorem.
\end{proof}

These technical results can be used to prove Theorem~\ref{thm:ERK-SSP-3-3} as
follows.
\begin{proof}[Proof of Theorem~\ref{thm:ERK-SSP-3-3}]
  The general solution of the order conditions for third order explicit
  Runge--Kutta methods with three stages is given by the two parameter family
  \begin{equation}
  \label{eq:rk33-two-parameters}
  \begin{aligned}
    a_{21} &= \alpha_2,
    &
    b_1 &= 1 + \frac{2 - 3 (\alpha_2 + \alpha_3)}{6 \alpha_2 \alpha_3},
    \\
    a_{31} &= \frac{3 \alpha_2 \alpha_3 (1 - \alpha_2) - \alpha_3^2}
                   {\alpha_2 (2 - 3 \alpha_2)},
    \qquad&
    b_2 &= \frac{3 \alpha_3 - 2}{6 \alpha_2 (\alpha_3 - \alpha_2)},
    \\
    a_{32} &= \frac{\alpha_3 (\alpha_3 - \alpha_2)}
                   {\alpha_2 (2 - 3 \alpha_2)},
    &
    b_3 &= \frac{2 - 3 \alpha_2}{6 \alpha_3 (\alpha_3 - \alpha_2)},
  \end{aligned}
  \end{equation}
  where $\alpha_2, \alpha_3 \neq 0$, $\alpha_2 \neq \alpha_3$, $\alpha_2 \neq 2/3$
  and the two one parameter families
  \begin{equation}
  \label{eq:rk33-one-parameter-1}
  \begin{aligned}
    a_{21} &= \frac{2}{3},
    &
    b_1 &= \frac{1}{4},
    \\
    a_{31} &= \frac{2}{3} - \frac{1}{4 \omega_3},
    \qquad&
    b_2 &= \frac{3}{4} - \omega_3,
    \\
    a_{32} &= \frac{1}{4 \omega_3},
    &
    b_3 &= \omega_3,
  \end{aligned}
  \end{equation}
  where $\alpha_2 = \alpha_3 = 2/3$ and
  \begin{equation}
  \label{eq:rk33-one-parameter-2}
  \begin{aligned}
    a_{21} &= \frac{2}{3},
    &
    b_1 &= \frac{1}{4} - \omega_3,
    \\
    a_{31} &= \frac{1}{4 \omega_3},
    \qquad&
    b_2 &= \frac{3}{4},
    \\
    a_{32} &= -\frac{1}{4 \omega_3},
    &
    b_3 &= \omega_3,
  \end{aligned}
  \end{equation}
  where $\alpha_3 = 0$, cf. \cite{ralston1962runge}. Thus, it suffices to check
  each case separately.

  Clearly, the necessary condition $a_{31} + a_{32} \geq \nicefrac{7}{8}$ of Lemma~\ref{lem:erk33-2}
  is violated for both one parameter families. Thus, it suffices to study the
  two parameter family \eqref{eq:rk33-two-parameters} for all possible cases.

  Due to Lemma~\ref{lem:A-b-geq-0}, all coefficients $a_{ij},b_i$ have to be
  non-negative for an SSP method. In particular, $\alpha_2 = a_{21} \geq 0$.
  Due to Lemma~\ref{lem:erk33-2}, $\alpha_3 = a_{31} + a_{32}$ has to satisfy
  $\nicefrac{7}{8} \leq \alpha_3 \leq \nicefrac{5}{4}$.
  \begin{itemize}
    \item
    $0 < \alpha_2 < \nicefrac{2}{3}$,
    $\nicefrac{7}{8} \leq \alpha_3 \leq \nicefrac{5}{4}$.

    In this case, $2 - 3 \alpha_2 > 0$ and
    \begin{equation}
      a_{31}
      =
      \underbrace{\frac{1}{\alpha_2 (2 - 3 \alpha_2)} \alpha_3}_{>0}
      \bigl(
        \underbrace{3 \alpha_2 (1 - \alpha_2)}_{\leq \nicefrac{3}{4}}
        - \underbrace{\alpha_3}_{\geq \nicefrac{7}{8}}
      \bigr)
      < 0.
    \end{equation}
    Thus, this case is excluded.

    \item
    $\nicefrac{2}{3} < \alpha_2$,
    $\nicefrac{7}{8} \leq \alpha_3 \leq \nicefrac{5}{4}$,
    $\alpha_2 \neq \alpha_3$.

    In this case, $2 - 3 \alpha_2 < 0$. Since
    \begin{equation}
      a_{32}
      =
      \underbrace{\frac{\alpha_3}{\alpha_2 (2 - 3 \alpha_2)}}_{< 0}
      (\alpha_3 - \alpha_2),
    \end{equation}
    the condition $a_{32} \geq 0$ is equivalent to $\alpha_3 < \alpha_2$.
    However, due to
    \begin{equation}
      b_2
      =
      \underbrace{\frac{3 \alpha_3 - 2}{6 \alpha_2}}_{> 0}
      \;
      \frac{1}{(\alpha_3 - \alpha_2)},
    \end{equation}
    $b_2 \geq 0$ requires $\alpha_3 > \alpha_2$, contradicting the requirement
    for $a_{32} \geq 0$, Hence, this case is also excluded.
  \end{itemize}
  This proves Theorem~\ref{thm:ERK-SSP-3-3}.
\end{proof}

\subsection{Schemes of at Least Second Order of Accuracy}

A generalisation of Theorem~\ref{thm:ERK-SSP-3-3} is
\begin{theorem}
\label{thm:ERK-SSP-3-2}
  There is no explicit Runge--Kutta method that
  \begin{itemize}
    \item is strong stability preserving with positive SSP coefficient,
    \item is of at least second order of accuracy \& has at most three stages,
    \item and is strongly stable for \eqref{eq:ode} for all smooth and semibounded
          $g$ with $\normLip{g} \leq L$.
  \end{itemize}
\end{theorem}

The basic idea of the proof of Theorem~\ref{thm:ERK-SSP-3-2} is the same as for
the proof of Theorem~\ref{thm:ERK-SSP-3-3}. However, the technical details are
a bit more complicated.

As before, the initial value problem \eqref{eq:ode} with \eqref{eq:ode-u1-r-u2}
will be used, where $r$ is a real parameter, $\alpha > 0$, and $u_1,u_2$ are real
valued functions.
\begin{lemma}
\label{lem:erk32-u1-r-u2}
  Applying an explicit three stage Runge--Kutta method with at least second order
  of accuracy given by the parameters \eqref{eq:rk33} to the ODE \eqref{eq:ode}
  with \eqref{eq:ode-u1-r-u2} yields $\norm{u_+}^2 - \norm{u_0}^2 = \dt^3 p(\dt)$,
  where $p(\dt)$ is a polynomial of the form
  \begin{multline}
    p(\dt)
    =
    \left(
      -1 + a_{21} - 2 a_{21} a_{31} b_3 + 2 a_{31}^2 b_3 + 4 a_{31} a_{32} b_3 + 2 a_{32}^2 b_3
    \right) r \alpha^3
    \\
    + \frac{1}{4} \left(
      1 + r^2 - 8 a_{21} a_{32} b_3 (2 - r^2 + a_{31} (-1 + 2 r^2) + a_{32} (-1 + 2 r^2))
    \right) \dt \alpha^4
    + \O(\dt^2).
  \end{multline}
\end{lemma}
Since the restrictions imposed by \eqref{eq:ode-u1-r-u2} do not seem to suffice
to prove Theorem~\ref{thm:ERK-SSP-3-2}, the initial value problem \eqref{eq:ode} with
\begin{equation}
\label{eq:ode-r-u1-u2}
\begin{aligned}
  u(t) = \vect{u_1(t) \\ u_2(t)},
  \quad
  g(u) = \alpha (r u_1 - u_2) \vect{-u_2 \\ u_1},
  \quad
  u_0 = \vect{1 \\ 0},
\end{aligned}
\end{equation}
will be used additionally, where $r$ is again a real parameter and $\alpha > 0$.
\begin{lemma}
\label{lem:erk32-r-u1-u2}
  Applying an explicit three stage Runge--Kutta method with at least second order
  of accuracy given by the parameters \eqref{eq:rk33} to the ODE \eqref{eq:ode}
  with \eqref{eq:ode-r-u1-u2} yields $\norm{u_+}^2 - \norm{u_0}^2 = \dt^3 p(\dt)$,
  where $p(\dt)$ is a polynomial of the form
  \begin{multline}
    p(\dt)
    =
    \left(
      -1 + a_{21} - 2 a_{21} a_{31} b_3 + 2 a_{31}^2 b_3 + 4 a_{31} a_{32} b_3 + 2 a_{32}^2 b_3
    \right) r^2 \alpha^3
    \\
    + \frac{1}{4} r^2 \left(
      1 + r^2 + 8 a_{21} a_{32} b_3 \bigl( 1 - 2 r^2 + (a_{31} + a_{32}) (-2 + r^2) \bigr)
    \right) \dt \alpha^4
    + \O(\dt^2).
  \end{multline}
\end{lemma}
Both Lemma~\ref{lem:erk32-u1-r-u2} and Lemma~\ref{lem:erk32-r-u1-u2} can be proved
by direct but tedious calculations and have been verified using
Mathematica~\citep{mathematica10}.

These technical results can be used to prove Theorem~\ref{thm:ERK-SSP-3-2} as
follows.
\begin{proof}[Proof of Theorem~\ref{thm:ERK-SSP-3-2}]
  For $s=3$ stages, the conditions for an order of accuracy $2$ are
  \begin{equation*}
  \stepcounter{equation}
  \tag{\theequation a}
  \label{eq:erk32-order}
    \sum_{j=1}^s b_j = 1,
    \quad
    \sum_{j,k=1}^s b_j a_{jk} = \frac{1}{2}.
  \end{equation*}
  As in Lemma~\ref{lem:erk33-2}, the coefficient of the constant term of the
  polynomial $p(\dt)$ in Lemma~\ref{lem:erk32-u1-r-u2} has to be non-positive
  for all $r \in \R$. Since
  $p(0) = (-1 + a_{21} - 2 a_{21} a_{31} b_3 + 2 a_{31}^2 b_3 + 4 a_{31} a_{32} b_3 + 2 a_{32}^2 b_3) r \alpha^3$, this implies
  \begin{equation*}
  \tag{\theequation b}
  \label{eq:erk32-ss-1}
    -1 + a_{21} - 2 a_{21} a_{31} b_3 + 2 a_{31}^2 b_3 + 4 a_{31} a_{32} b_3 + 2 a_{32}^2 b_3
    =
    0.
  \end{equation*}
  Furthermore, Lemma~\ref{lem:erk32-u1-r-u2} (i.e. the right hand side
  \eqref{eq:ode-u1-r-u2}) yields the condition
  \begin{equation*}
  \tag{\theequation c}
  \label{eq:erk32-ss-2}
    \forall r \in \R\colon
    \quad
    1 + r^2 - 8 a_{21} a_{32} b_3 \bigl( 2 - r^2 + (a_{31} + a_{32}) (-1 + 2 r^2) \bigr)
    \leq 0.
  \end{equation*}
  Similarly, Lemma~\ref{lem:erk32-r-u1-u2} (i.e. the right hand side
  \eqref{eq:ode-r-u1-u2}) yields the condition
  \begin{equation*}
  \tag{\theequation d}
  \label{eq:erk32-ss-3}
    \forall r \in \R\colon
    \quad
    r^2 \left(
      1 + r^2 + 8 a_{21} a_{32} b_3 \bigl( 1 - 2 r^2 + (a_{31} + a_{32}) (-2 + r^2) \bigr)
    \right)
    \leq 0.
  \end{equation*}
  As in the proof of Lemma~\ref{lem:erk33-2}, $\alpha > 0$ can be adapted to $r$
  such that the local Lipschitz constant of $g$ near $u_0$ is as small as desired.
  Outside of such a neighbourhood, $g$ can be modified to keep this Lipschitz
  constant (Kirszbraun’s theorem).

  Finally, Lemma~\ref{lem:A-b-geq-0} yields the conditions
  \begin{equation*}
  \tag{\theequation e}
  \label{eq:erk32-A-b-geq-0}
    a_{21} \geq 0, \;
    a_{31} \geq 0, \;
    a_{32} \geq 0, \;
    b_1 \geq 0, \;
    b_2 \geq 0, \;
    b_3 \geq 0.
  \end{equation*}
  Applying the function \texttt{Reduce} of Mathematica~\citep{mathematica10} to
  equations \eqref{eq:erk32-order} to \eqref{eq:erk32-A-b-geq-0} yields the
  single possibility
  \begin{equation}
  \label{eq:erk32-result}
    a_{21} = \frac{1}{2}, \;
    a_{31} = 0, \;
    a_{32} = 1, \;
    b_1 = \frac{1}{4}, \;
    b_2 = \frac{1}{2}, \;
    b_3 = \frac{1}{4}.
  \end{equation}
  This scheme is not strongly stable for the ODE \eqref{eq:ode} with $g(u) = \L u$,
  where $\L$ is a general linear and skew-symmetric operator. This can be seen by
  considering the classical stability region of this scheme. Indeed, the stability
  function is $R(z) = \det(\I - z A + z \mathbbm{1} b^T) = 1 + z + \frac{1}{2} z^2 + \frac{1}{8} z^3$. Considering $z = y \i$ with $y \in \R$ yields
  \begin{equation}
    \abs{ R(y \i) }^2
    =
    \abs{ 1 + y \i - \frac{1}{2} y^2 - \frac{1}{8} y^3 \i }^2
    =
    \left( 1 - \frac{1}{2} y^2 \right)^2
    + \left( y - \frac{1}{8} y^3 \right)^2
    =
    1 + \frac{1}{64} y^6.
  \end{equation}
  Hence, $\abs{ R(y \i) } > 1$ for $y \neq 0$ and the scheme \eqref{eq:erk32-result}
  is not strongly stable in general.
  This proves Theorem~\ref{thm:ERK-SSP-3-2}.
\end{proof}

\section{Some Known Explicit Methods}
\label{sec:ssprks2-ssprks3}

Since the general explicit Runge--Kutta with more than three stages has an increased
number of coefficients, an approach similar to the one in the previous section
is not really feasible. Therefore, some specific methods with up to ten stages will
be studied in this section.

As before, the impossibility results will be obtained using some specifically
designed test problems. In the following, the ODE \eqref{eq:ode} with
\begin{equation}
\label{eq:ode-u1-u2}
\begin{aligned}
  u(t) = \vect{u_1(t) \\ u_2(t)},
  \quad
  g(u) = (u_1 - u_2) \vect{-u_2 \\ u_1},
  \quad
  u_0 = \vect{1 \\ 0},
\end{aligned}
\end{equation}
will be used, i.e. \eqref{eq:ode-u1-r-u2} or \eqref{eq:ode-r-u1-u2} with $r=1$.

\subsection{Second Order Methods}

The unique second order explicit SSP Runge--Kutta method SSPRK(s,2) with $s \geq 2$
stages and optimal (maximal) SSP coefficient is given by the Butcher coefficients
\citep[Theorem~9.3]{kraaijevanger1991contractivity}
\begin{equation}
  a_{i,j} = \frac{1}{s-1},
  \quad
  b_i = \frac{1}{s},
  \qquad
  \forall i,j \in \set{1,\dots,s}, j < i.
\end{equation}
These schemes can be implemented in a low storage form as \citep{ketcheson2008highly}
\begin{equation}
\label{eq:ssprks2}
\begin{aligned}
  u_k &= u_{k-1} + \frac{\dt}{s-1} g(u_{k-1}), \qquad k \in \set{1,\dots,s},
  \\
  u_+ &= \frac{s-1}{s} u_s + \frac{1}{s} u_0.
\end{aligned}
\end{equation}
\begin{theorem}
\label{thm:ssprks2}
  The second order explicit SSP Runge--Kutta methods SSPRK(s,2), $s \geq 2$, of
  \cite[Theorem~9.3]{kraaijevanger1991contractivity} are not strongly stable for
  the ODE \eqref{eq:ode} for all smooth and semibounded $g$ with $\normLip{g} \leq L$.
\end{theorem}

In order to prove Theorem~\ref{thm:ssprks2}, the following technical result will
be used.
\begin{lemma}
\label{lem:ssprks2}
  For the ODE \eqref{eq:ode} with parameters \eqref{eq:ode-u1-u2}, the stages
  $u_k$, $k \in \set{0, \dots, s}$, in \eqref{eq:ssprks2} satisfy
  \begin{equation}
  \begin{aligned}
    u_{k,1}
    &=
    1
    - \frac{k (k-1)}{2} \biggl( \frac{\dt}{s-1} \biggr)^2
    + \frac{k (k-1)^2}{2} \biggl( \frac{\dt}{s-1} \biggr)^3
    - \frac{(k+1) k (k-1) (k-2)}{12} \biggl( \frac{\dt}{s-1} \biggr)^4
    + \O(\dt^5),
    \\
    u_{k,2}
    &=
    k \biggl( \frac{\dt}{s-1} \biggr)
    - \frac{k (k-1)}{2} \biggl( \frac{\dt}{s-1} \biggr)^2
    - \frac{k (k-1) (k-2)}{6} \biggl( \frac{\dt}{s-1} \biggr)^3
    \\&\phantom{=}\qquad
    + \frac{(5k-7) k (k-1) (k-2)}{12} \biggl( \frac{\dt}{s-1} \biggr)^4
    + \O(\dt^5).
  \end{aligned}
  \end{equation}
\end{lemma}
\begin{proof}
  Since $u_0 = (1, 0)$, the result is true for $k=0$. Assuming the result holds
  for $k$, inserting \eqref{eq:ode-u1-u2} into \eqref{eq:ssprks2} proves the
  result for $k+1$ and thus for general $k \in \set{0,\dots,s}$.
\end{proof}

\begin{proof}[Proof of Theorem~\ref{thm:ssprks2}]
  Consider the ODE \eqref{eq:ode} with parameters \eqref{eq:ode-u1-u2}.
  Using Lemma~\ref{lem:ssprks2},
  \begin{equation}
  \begin{aligned}
    \norm{u_+}^2 - \norm{u_0}^2
    &=
    \biggl( \frac{s-1}{s} u_{s,1} + \frac{1}{s} \biggr)^2
    + \biggl( \frac{s-1}{s} u_{s,2} \biggr)^2
    - 1
    \\
    &=
    \biggl(
      1
      - \frac{1}{2} \dt^2
      + \frac{1}{2} \dt^3
      - \frac{(s+1) (s-2)}{12 (s-1)^2} \dt^4
    \biggr)^2
    - 1 + \O(\dt^5)
    \\&\quad
    + \biggl(
      \dt
      - \frac{1}{2} \dt^2
      - \frac{s-2}{6 (s-1)} \dt^3
      + \frac{(5s-7) (s-2)}{12 (s-1)^2} \dt^4
    \biggr)^2
    =
    \frac{s+1}{6 (s-1)^2} \dt^4 + \O(\dt^5).
  \end{aligned}
  \end{equation}
  Since $(s+1)/(6 (s-1)^2) > 0$ for $s \geq 2$, $\norm{u_+}^2 > \norm{u_0}^2$
  for small $\dt > 0$ and Theorem~\ref{thm:ssprks2} is proved by applying the
  same arguments as in the proofs given hitherto to reduce the Lipschitz
  constant as desired.
\end{proof}

\subsection{Third Order Methods}

There is also a family of third order SSP methods with optimal SSP coefficient
and $s = n^2$ stages for $n \in \N$, $n \geq 2$ \citep[Theorem~3]{ketcheson2008highly}.
This family contains the method SSPRK(4,3)
of \citet[Theorem~9.5]{kraaijevanger1991contractivity}. The
schemes of this family can be implemented in low storage form as \citep{ketcheson2008highly}
\begin{equation}
\label{eq:ssprkn23}
\begin{aligned}
  u_k &= u_{k-1} + \frac{\dt}{n (n-1)} g(u_{k-1}), &&k \in \set{1,\dots,\frac{n(n+1)}{2}},
  \\
  v_0 &= \frac{n}{2n-1} u_{\frac{(n-1)(n-2)}{2}} + \frac{n-1}{2n-1} u_{\frac{n(n+1)}{2}},
  \\
  v_k &= v_{k-1} + \frac{\dt}{n (n-1)} g(v_{k-1}), &&k \in \set{1,\dots,\frac{n(n-1)}{2}},
  \\
  u_+ &= v_{\frac{n(n-1)}{2}}.
\end{aligned}
\end{equation}
\begin{theorem}
\label{thm:ssprkn23}
  The third order explicit SSP Runge--Kutta methods SSPRK($n^2$,3), $n \geq 2$,
  of \citep[Theorem~3]{ketcheson2008highly} are not strongly stable for the ODE
  \eqref{eq:ode} for all smooth and semibounded $g$ with $\normLip{g} \leq L$.
\end{theorem}

In order to prove Theorem~\ref{thm:ssprkn23}, the following technical results will
be used.
\begin{lemma}
\label{lem:ssprkn23-u}
  For the ODE \eqref{eq:ode} with parameters \eqref{eq:ode-u1-u2}, the stages
  $u_k$, $k \in \set{0, \dots, \frac{n(n+1)}{2}}$, in \eqref{eq:ssprkn23} satisfy
  \begin{equation}
  \begin{aligned}
    u_{k,1}
    &=
    1
    - \frac{k (k-1)}{2 n^2 (n-1)^2} \dt^2
    + \frac{k (k-1)^2}{2 n^3 (n-1)^3} \dt^3
    \\&\qquad
    - \frac{(k+1) k (k-1) (k-2)}{12 n^4 (n-1)^4} \dt^4
    - \frac{(3k-7) k (k-1)^2 (k-2)}{12 n^5 (n-1)^5} \dt^5
    + \O(\dt^6),
    \\
    u_{k,2}
    &=
    \frac{k}{n (n-1)} \dt
    - \frac{k (k-1)}{2 n^2 (n-1)^2} \dt^2
    - \frac{k (k-1) (k-2)}{6 n^3 (n-1)^3} \dt^3
    \\&\qquad
    + \frac{(5k-7) k (k-1) (k-2)}{12 n^4 (n-1)^4} \dt^4
    - \frac{(13 k^2 - 41 k + 26) k (k-1) (k-2)}{60 n^5 (n-1)^5} \dt^5
    + \O(\dt^6).
  \end{aligned}
  \end{equation}
\end{lemma}
\begin{proof}
  Since $u_0 = (1, 0)$, the result is true for $k=0$. Assuming the result holds
  for $k$, inserting \eqref{eq:ode-u1-u2} into \eqref{eq:ssprkn23} proves the
  result for $k+1$ and thus for general $k \in \set{0,\dots,\frac{n(n+1)}{2}}$.
\end{proof}

\begin{lemma}
\label{lem:ssprkn23-v}
  For the ODE \eqref{eq:ode} with parameters \eqref{eq:ode-u1-u2}, the stages
  $v_{k}$, $k \in \set{0, \dots, \frac{n(n-1)}{2}}$, in \eqref{eq:ssprkn23} satisfy
  \begin{align*}
  \stepcounter{equation}\tag{\theequation}
    v_{k,1}
    &=
    1
    + \frac{-4 k^2+k (-4 n^2+4 n+4)-n (n^3-2 n^2+3 n-2)}{8
    (n-1)^2 n^2} \dt^2
    \\&\qquad
    + \Bigl( 8 k^3+4 k^2 (3 n^2-3 n-4)+2 k (3 n^4-6 n^3-n^2+4
    n+4)
    \\&\qquad\qquad
    +n (n^5-3 n^4+11 n^3-17 n^2+4 n+4)
    \Bigr) \frac{1}{16 (n-1)^3 n^3} \dt^3
    \\&\qquad
    + \Bigl( 16 k^4+32 k^3 (n^2-n-1)+8 k^2 (3 n^4-6 n^3-3 n^2+6
    n-2)
    \\&\qquad\qquad
    +8 k (n^6-3 n^5+5 n^3+3 n^2-6 n+4)
    \\&\qquad\qquad
    +n (n^7-4 n^6+26 n^5-64 n^4+57 n^3-12 n^2-20 n+16)
    \Bigr) \frac{1}{192 (n-1)^4 n^4} \dt^4
    \\&\qquad
    + \Bigl( 96 k^5+16 k^4 (15 n^2-15 n-38)+16 k^3 (15 n^4-30
    n^3-41 n^2+56 n+86)
    \\&\qquad\qquad
    +8 k^2 (15 n^6-45 n^5+27 n^4+21 n^3+96
    n^2-114 n-164)
    \\&\qquad\qquad
    +2 k (15 n^8-60 n^7+178 n^6-324 n^5+11 n^4+448
    n^3-284 n^2+16 n+224)
    \\&\qquad\qquad
    +n (3 n^9-15 n^8+112 n^7-358 n^6+247
    n^5+449 n^4-354 n^3-428 n^2+120 n+224)
    \Bigr)
    \\&\qquad\qquad
    \frac{1}{384 (n-1)^5 n^5} \dt^5
    + \O(\dt^6),
  \end{align*}
  and
  \begin{align*}
  \stepcounter{equation}\tag{\theequation}
    v_{k,2}
    &=
    \frac{2 k+n^2-n}{2 (n-1) n} \dt
    + \frac{-4 k^2+k (-4 n^2+4 n+4)-n (n^3-2 n^2+3 n-2)}{8
    (n-1)^2 n^2} \dt^2
    \\&\qquad
    - \Bigl(8 k^3+12 k^2 (n^2-n-2)+2 k (3 n^4-6 n^3+3
    n^2+8)
    \\&\qquad\qquad
    +n (n^5-3 n^4+9 n^3-13 n^2-2 n+8)
    \Bigr) \frac{1}{48 (n-1)^3 n^3} \dt^3
    \\&\qquad
    + \Bigl(80 k^4+32 k^3 (5 n^2-5 n-11)+8 k^2 (15 n^4-30 n^3-27
    n^2+42 n+62)
    \\&\qquad\qquad
    +8 k (5 n^6-15 n^5+30 n^4-35 n^3+21 n^2-6
    n-28)
    \\&\qquad\qquad
    +n (5 n^7-20 n^6+106 n^5-248 n^4+69 n^3+252 n^2-52 n-112)
    \Bigr) \frac{1}{192 (n-1)^4 n^4} \dt^4
    \\&\qquad
    - \Bigl(416 k^5+80 k^4 (13 n^2-13 n-32)+80 k^3 (13 n^4-26
    n^3-39 n^2+52 n+70)
    \\&\qquad\qquad
    +40 k^2 (13 n^6-39 n^5+15 n^4+35 n^3+98
    n^2-122 n-128)
    \\&\qquad\qquad
    +2 k (65 n^8-260 n^7+950 n^6-1940 n^5+725
    n^4+1480 n^3-1500 n^2+480 n+832)
    \\&\qquad\qquad
    +n (13 n^9-65 n^8+490
    n^7-1570 n^6+1165 n^5+1727 n^4-1508 n^3-1564 n^2
    \\&\qquad\qquad\qquad
    +480 n+832)
    \Bigr) \frac{1}{1920 (n-1)^5 n^5} \dt^5
    + \O(\dt^6).
  \end{align*}
\end{lemma}
\begin{proof}
  Using Lemma~\ref{lem:ssprkn23-u}, the result can be verified for $k=0$. Assuming
  that the result holds for $k$, inserting \eqref{eq:ode-u1-u2} into \eqref{eq:ssprkn23}
  proves the result for $k+1$ and thus for general $k \in \set{0,\dots,\frac{n(n-1)}{2}}$.
\end{proof}

\begin{proof}[Proof of Theorem~\ref{thm:ssprkn23}]
  Using Lemma~\ref{lem:ssprkn23-v},
  \begin{equation}
  \begin{aligned}
    \norm{u_+}^2 - \norm{u_0}^2
    &=
    \frac{n^2 - n- 2}{12 n^2 (n-1)^2} \dt^4
    + \frac{n^2 - n + 3}{6 n^2 (n-1)^2} \dt^5
    + \O(\dt^6).
  \end{aligned}
  \end{equation}
  For $n = 2$, $\frac{n^2 - n- 2}{12 n^2 (n-1)^2} = 0$ and
  $\frac{n^2 - n + 3}{6 n^2 (n-1)^2} > 0$. For $n \geq 3$,
  $\frac{n^2 - n- 2}{12 n^2 (n-1)^2} > 0$.
  Thus, $\norm{u_+}^2 > \norm{u_0}^2$ for small $\dt > 0$ and Theorem~\ref{thm:ssprkn23}
  is proved by applying the same arguments as in the proofs given hitherto
  to reduce the Lipschitz constant as desired.
\end{proof}

\subsection{Ten Stage, Fourth Order Method SSPRK(10,4)}
\label{subsec:ssprk104}

The explicit strong-stability preserving method SSPRK(10,4) of \citet{ketcheson2008highly}
is given by the Butcher tableau
\begin{equation}
\begin{aligned}
  \begin{array}{c|cccccccccc}
  0 &  &  &  &  &  &  &  &  &  & \\
  \nicefrac{1}{6} & \nicefrac{1}{6} &  &  &  &  &  &  &  &  & \\
  \nicefrac{1}{3} & \nicefrac{1}{6} & \nicefrac{1}{6} &  &  &  &  &  &  &  & \\
  \nicefrac{1}{2} & \nicefrac{1}{6} & \nicefrac{1}{6} & \nicefrac{1}{6} &  &  &  &  &  &  & \\
  \nicefrac{2}{3} & \nicefrac{1}{6} & \nicefrac{1}{6} & \nicefrac{1}{6} & \nicefrac{1}{6} &  &  &  &  &  & \\
  \nicefrac{1}{3} & \nicefrac{1}{15} & \nicefrac{1}{15} & \nicefrac{1}{15} & \nicefrac{1}{15} & \nicefrac{1}{15} &  &  &  &  & \\
  \nicefrac{1}{2} & \nicefrac{1}{15} & \nicefrac{1}{15} & \nicefrac{1}{15} & \nicefrac{1}{15} & \nicefrac{1}{15} & \nicefrac{1}{6} &  &  &  & \\
  \nicefrac{2}{3} & \nicefrac{1}{15} & \nicefrac{1}{15} & \nicefrac{1}{15} & \nicefrac{1}{15} & \nicefrac{1}{15} & \nicefrac{1}{6} & \nicefrac{1}{6} &  &  & \\
  \nicefrac{5}{6} & \nicefrac{1}{15} & \nicefrac{1}{15} & \nicefrac{1}{15} & \nicefrac{1}{15} & \nicefrac{1}{15} & \nicefrac{1}{6} & \nicefrac{1}{6} & \nicefrac{1}{6} &  & \\
  1 & \nicefrac{1}{15} & \nicefrac{1}{15} & \nicefrac{1}{15} & \nicefrac{1}{15} & \nicefrac{1}{15} & \nicefrac{1}{6} & \nicefrac{1}{6} & \nicefrac{1}{6} & \nicefrac{1}{6} & \\
  \hline
  & \nicefrac{1}{10} & \nicefrac{1}{10} & \nicefrac{1}{10} & \nicefrac{1}{10} & \nicefrac{1}{10} & \nicefrac{1}{10} & \nicefrac{1}{10} & \nicefrac{1}{10} & \nicefrac{1}{10} & \nicefrac{1}{10}\\
  \end{array}
\end{aligned}
\end{equation}
which is not sparse in the sense of many zeros, but ``data sparse'' in the sense
of a clear structure with few different values of the entries. This results in
a sparse Shu-Osher form and the low-storage implementation
\begin{equation}
\begin{gathered}
\begin{aligned}
  u_1
  &:=
  u_0,
  &
  u_i
  &:=
  u_{i-1} + \frac{\Delta t}{6} g(u_{i-1}),
  \; i \in \set{2,3,4,5},
  \\
  u_6
  &:=
  \frac{3}{5} u_0 + \frac{2}{5} \left( u_5 + \frac{\Delta t}{6} g(u_5) \right),
  &
  u_i
  &:=
  u_{i-1} + \frac{\Delta t}{6} g(u_{i-1}),
  \; i \in \set{7,8,9,10},
\end{aligned}
\\
  u_+
  :=
  \frac{1}{25} u_0
  + \frac{9}{25} \left( u_5 + \frac{\Delta t}{6} g(u_5) \right)
  + \frac{3}{5} \left( u_{10} + \frac{\Delta t}{6} g(u_{10}) \right).
\end{gathered}
\end{equation}

\begin{theorem}
\label{thm:ssprk104}
  The ten stage, fourth order, explicit strong stability preserving method SSPRK(10,4)
  of \citet{ketcheson2008highly} is not strongly stable for the ODE \eqref{eq:ode}
  for all smooth and semibounded $g$ with $\normLip{g} \leq L$.
\end{theorem}
\begin{proof}
  Consider the ODE \eqref{eq:ode} with parameters \eqref{eq:ode-u1-u2}. A lengthy
  calculation that has been verified using Mathematica~\citep{mathematica10}
  shows that $\norm{u_+}^2 - \norm{u_0}^2 = \dt^6 p(\dt)$, where $p(\dt)$ is a
  polynomial in $\dt$ with
  \begin{equation}
    p(\dt)
    =
    \frac{23}{3240} - \frac{1}{240} \dt - \frac{161}{29160} \dt^2 + \O(\dt^3).
  \end{equation}
  Hence, there is some $\tau > 0$ such that $\norm{u_+}^2 - \norm{u_0}^2 > 0$
  for all $\dt \in (0,\tau)$. Thus, SSPRK(10,4) is not strongly stable for this
  test problem. The same arguments as in the proofs given hitherto can be used
  to reduce the Lipschitz constant as desired.
\end{proof}

\section{Three Stage, Third Order Method SSPRK(3,3)}
\label{sec:ssprk33}

The third-order explicit strong stability preserving Runge--Kutta method SSPRK(3,3)
with three stages given by \citet{shu1988efficient} is determined by the Butcher
tableau
\begin{equation}
\begin{aligned}
  \begin{array}{c|ccc}
  0 &  & &  \\
  1 & 1 & & \\
  \nicefrac{1}{2} & \nicefrac{1}{4} & \nicefrac{1}{4} & \\
  \hline
  & \nicefrac{1}{6} & \nicefrac{1}{6} & \nicefrac{2}{3} \\
  \end{array}
\end{aligned}
\end{equation}
and can be represented using the Shu-Osher form
\begin{equation}
\label{eq:ssprk33}
\begin{aligned}
  u_1
  &:=
  u_0,
  \\
  u_2
  &:=
  u_1 + \dt g(u_1),
  \\
  u_3
  &:=
  \frac{3}{4} u_0 + \frac{1}{4} \left( u_2 + \dt g(u_2) \right)
  =
  u_0 + \frac{1}{4} \dt g(u_1) + \frac{1}{4} \dt g(u_2),
  \\
  u_+
  &:=
  \frac{1}{3} u_0 + \frac{2}{3} \left( u_3 + \dt g(u_3) \right)
  =
  u_0 + \frac{1}{6} \dt g(u_1) + \frac{1}{6} \dt g(u_2) + \frac{2}{3} \dt g(u_3).
\end{aligned}
\end{equation}
Theorem~\ref{thm:ERK-SSP-3-3} implies that SSPRK(3,3) is not strongly stable for
general semibounded $g$. Since this method is often used, it is considered for
further stability investigations in this section. In particular, the following
properties will be studied.
\begin{itemize}
  \item
  The classical fourth order, four stage explicit Runge--Kutta method RK(4,4) is
  not strongly stable for general semibounded and linear $g$. However, the method
  given by two consecutive steps of RK(4,4) is strongly stable for such $g$, cf.
  \cite{sun2017stability}. Thus, it is of interest whether something similar is
  true for SSPRK(3,3) for general nonlinear semibounded $g$.

  \item
  Even if SSPRK(3,3) is not strongly stable after a finite number of steps, the
  increase of the norm might still be bounded, cf. \cite{hundsdorfer2009stepsize,
  hundsdorfer2011boundedness, hundsdorfer2011special} for investigations of such
  a property when the explicit Euler method is assumed to be strongly stable.
  Although boundedness and monotonicity are equivalent in this context for a
  large class of Runge--Kutta methods \citep{hundsdorfer2011boundedness}, this
  result cannot be applied here directly. Hence, it is of interest to study the
  behaviour of SSPRK(3,3) for semibounded $g$.
\end{itemize}
Of course, both properties are related in some way. In particular, it will be proven
that there are semibounded $g$ such that the norm of a numerical solution obtained
using SSPRK(3,3) is monotonically increasing and unbounded. This implies that
SSPRK(3,3) cannot be strongly stable after any finite number of steps.

\begin{theorem}
\label{thm:ssprk33}
  There are smooth and semibounded $g$  with $\normLip{g} \leq L$ and the property
  that the application of
  the three stage, third order explicit strong stability preserving Runge--Kutta
  method SSPRK(3,3) of \cite{shu1988efficient} to the ODE \eqref{eq:ode} yields
  a numerical solution $\unum$ such that the sequence
  $\bigl( \norm{\unum(n \dt)}^2 \bigr)_{n \in \N_0}$
  is monotonically increasing and unbounded for every $\dt > 0$.
\end{theorem}
As the proofs of the previous results, this one is based on the explicit construction
of carefully designed test problems. Here, the ODE \eqref{eq:ode} with
\begin{equation}
\label{eq:ode-u1-u2-norm}
\begin{aligned}
  u(t) = \vect{u_1(t) \\ u_2(t)},
  \quad
  g(u) = \frac{\alpha}{\norm{u}^2} \vect{-u_2 \\ u_1},
\end{aligned}
\end{equation}
will be considered for $\alpha > 0$. Since the norm of the numerical solution will
be shown to increase monotonically and the norm of a smooth solutions remains constant,
the function $g$ could be modified to remove the singularity at zero and keep the
Lipschitz constant as small as desired for a suitable choice of $\alpha > 0$.
Indeed, for $\norm{u},\norm{v} \geq 1$,
\begin{equation}
\begin{aligned}
  \norm{g(u) - g(v)}
  &=
  \frac{\alpha}{\norm{u} \norm{v}}
    \norm{\,\norm{v} \frac{u}{\norm{u}} - \norm{u} \frac{u}{\norm{u}}
            + \norm{u} \frac{u}{\norm{u}} - \norm{u} \frac{v}{\norm{v}} }
  \\
  &\leq
  \frac{\alpha}{\norm{u} \norm{v}} \left(
    \abs{\,\norm{v} - \norm{u}}
    + \norm{v}^{-1} \norm{\,\norm{v} u - \norm{v} v + \norm{v} v - \norm{u} v}
  \right)
  \\
  &\leq
  \frac{3 \alpha}{\norm{u} \norm{v}} \norm{u-v},
\end{aligned}
\end{equation}
showing that the Lipschitz constant of $g$ is bounded from above by $3 \alpha$
in $\R^2 \setminus B_1(0)$.
\begin{proof}[Proof of Theorem~\ref{thm:ssprk33}]
  Consider one step of SSPRK(3,3) from $u_0 = (u_{0,1}, u_{0,2})$ to
  $u_+ = (u_{+,1}, u_{+,2})$. A lengthy calculation that has been verified using
  Mathematica~\citep{mathematica10} yields
  \begin{equation}
  \begin{aligned}
    \norm{u_+}^2 - \norm{u_0}^2
    &=
    f_{\alpha \dt}\bigl( \norm{u_0}^2 \bigr),
    \\
    f_{\dt}(x)
    &=
    \dt^4
    \frac{\dt^4 + 196 \dt^2 x^2 + 240 x^4}
         {36 x \bigl( \dt^2 + x^2 \bigr)
                            \bigl( \dt^4 + 12 \dt^2 x^2 + 16 x^4 \bigr)},
  \end{aligned}
  \end{equation}
  for $\dt > 0$. In particular, the squared norms of the numerical solution
  $\norm{ \unum(n \dt) }^2$ are given recursively by
  \begin{equation}
    \norm{ \unum((n+1) \dt) }^2
    =
    \norm{ \unum(n \dt) }^2
    + f_{\alpha \dt}\bigl( \norm{ \unum(n \dt) }^2 \bigr).
  \end{equation}
  Since $f_{\dt}(x) > 0$ for every $\dt > 0$ and $x > 0$, they increase
  monotonically with $n$ if, e.g. $\norm{u_0} = 1$.
  If they were bounded, the sequence of squared norms would have a (positive)
  limit $x$ satisfying $f_{\alpha \dt}(x) = 0$ which is impossible for $\alpha \dt > 0$.
\end{proof}

Numerical solutions of \eqref{eq:ode-u1-u2-norm} with $\alpha = 1$ and initial
condition $u_0 = (1, 0)$ have been computed using SSPRK(3,3) implemented in
DifferentialEquations.jl \citep{rackauckas2017differentialequations} in Julia \texttt{v1.1.0}
\citep{bezanson2017julia} using using floating point numbers with extended precision
(\texttt{BigFloat} with \texttt{setprecision(500)}).
As visualised in Figure~\ref{fig:ssprk33}, the energy increases monotonically
for every time step $\dt > 0$, in accordance with Theorem~\ref{thm:ssprk33}.
\begin{figure}[!ht]
\centering
\captionsetup{aboveskip=-4pt}
  \includegraphics[width=0.6\textwidth]{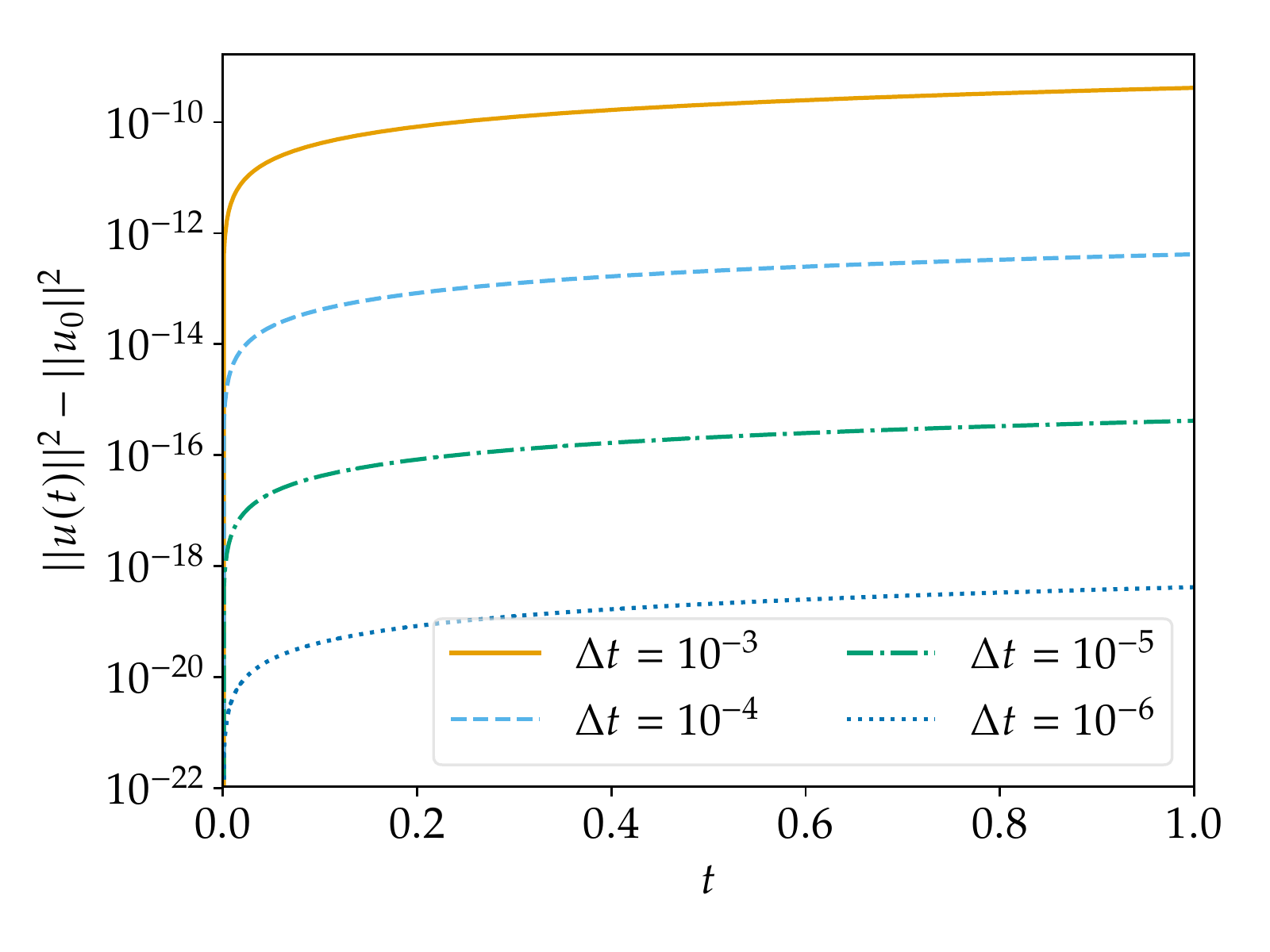}
  \caption{Evolution of the energy of numerical solutions computed using SSPRK(3,3)
           with different time steps $\dt$.}
  \label{fig:ssprk33}
\end{figure}

\begin{remark}
  In applications, only finite final times $T > 0$ are relevant. Hence, it can be
  interesting whether a bound of the form
  \begin{equation}
    \forall n \in \N, n \dt \leq T\colon \quad
    \norm{u(n \dt)} \leq c_T
  \end{equation}
  holds for some constant $c_T > 0$ depending on $T$. For explicit schemes,
  a useful bound seems to require an additional restriction of the time step $\dt$
  because of stability issues. Choosing $0 < \dt < \dt_{\mathrm{max}}$ small enough,
  such a bound will trivially hold if the scheme converges. However, it does not
  seem to be trivial to guarantee good estimates of $c_T$ and $\dt_{\mathrm{max}}$
  in general.
\end{remark}

\section{Ten Stage, Fourth Order Method SSPRK(10,4) and the Transport Equation}
\label{sec:ssprk104}

In this section, the method SSPRK(10,4) of \cite{ketcheson2008highly} described
in section~\ref{subsec:ssprk104} will be used to integrate a semidiscretisation
of a hyperbolic conservation law in time. It will be demonstrated numerically
that the energy (squared norm) of the solution increases for a wide range of
positive time steps.

Consider the linear advection equation with periodic boundary conditions
\begin{equation}
\label{eq:const-lin-adv}
\begin{aligned}
  \partial_t u(t,x) + \partial_x u(t,x) &= 0, && t \in (0,T), x \in (x_L,x_R),
  \\
  u(0,x) &= u_0(x), && x \in (x_L,x_R),
  \\
  u(t,x_L) &= u(t,x_R), && t \in (0,T),
\end{aligned}
\end{equation}
and the initial condition $u_0(x) = -\sin(\pi x)$ in the domain $(x_L,x_R) = (-1,1)$.
Using the $L^2$ entropy $U(u) = \frac{1}{2} u^2$, the entropy flux is
$F(u) = \frac{1}{2} u^2$ and the flux potential is $\psi(u) = \frac{1}{2} u^2$.
Smooth solutions fulfil $\norm{u(t)}_2^2 = \norm{u_0}_2^2$ and the entropy inequality
\begin{equation}
  \partial_t u(t,x)^2 + \partial_x u(t,x)^2 \leq 0
\end{equation}
yields $\norm{u(t)}_2^2 \leq \norm{u_0}_2^2$ for general solutions, cf.
\cite{tadmor1987numerical, tadmor2003entropy}.

Recently, \citet{abgrall2018general} proposed a general method to make numerical
schemes entropy conservative/stable. He described this approach using residual
distribution schemes and explains how some other frameworks can be recast in
this way, see also \cite{abgrall2017some, abgrall2018connection}. For nodal
discontinuous Galerkin (DG) schemes in one space dimension, this approach will be
adapted in the following.

Consider a general polynomial collocation approach using $p+1$ nodes and polynomials
of degree $\leq p$ in an element $[x_{i-1}, x_{i}]$. Besides the choice of the nodes,
the main ingredients are
\begin{itemize}
  \item
  a mass matrix $\mat{M}$, approximating the $L^2$ scalar product via
  $\int_{x_{i-1}}^{x_{i}} u(x) v(x) \dif x = \scp{u}{v}_{L^2} \approx
  \scp{\vec{u}}{\vec{v}}_M = \vec{u}^T \mat{M} \vec{v}$.

  \item
  a derivative matrix $\mat{D}$, approximating the derivative $\partial_x u
  \approx \mat{D} \vec{u}$.

  \item
  a restriction operator $\mat{R}$, performing interpolation to the boundary nodes
  $x_{i-1}, x_{i}$ via $\mat{R} \vec{u} = (u_L, u_R)^T$.

  \item
  a diagonal boundary matrix $\mat{B} = \diag{-1, 1}$, giving the difference of boundary
  values as in the fundamental theorem of calculus.
\end{itemize}

If the mass matrix is exact for polynomials of degree $\leq 2p-1$, the
summation by parts property
\begin{equation}
\label{eq:SBP}
  \mat{M} \mat{D} + \mat{D}[^T] \mat{M} = \mat{R}[^T] \mat{B} \mat{R}
\end{equation}
will be satisfied, cf. \cite{hicken2013summation, gassner2013skew, fernandez2014generalized}.
In that case, semidiscrete stability can be proven in many cases, as described in
the review articles of \citet{svard2014review, fernandez2014review} and references
cited therein.

A nodal DG semidiscretisation of the advection equation \eqref{eq:const-lin-adv}
will be performed as follows. At first, the domain $(x_L,x_R)$ is divided
uniformly into $N$ non-overlapping elements. Each element is mapped via an
affine-linear mapping to the reference element $(-1,1)$ and all computations are
performed there. On each element, the semidiscretisation is
\begin{equation}
\label{eq:const-lin-adv-DG-standard}
  \partial_t \vec{u}
  + \mat{D} \vec{u}
  =
  - \mat{M}[^{-1}] \mat{R}[^T] \mat{B} \left( \vecfnum - \mat{R} \vec{u} \right),
\end{equation}
where the numerical flux will be the central flux $\fnum(u_-,u_+) = \frac{u_-+u_+}{2}$.
This flux is entropy conservative for the $L^2$ entropy $U(u) = \frac{u^2}{2}$.
Thus, if SBP operators are used, e.g. via bases on Gauss or Lobatto Legendre nodes,
the resulting semidiscretisation is entropy conservative. Indeed, using SBP
operators yields
\begin{equation}
\begin{aligned}
  \vec{u}^T \mat{M} \partial_t \vec{u}
  &=
  - \vec{u}^T \mat{M} \mat{D} \vec{u}
  - \vec{u}^T \mat{R}[^T] \mat{B} \left( \vecfnum - \mat{R} \vec{u} \right)
  \\
  &=
  - \frac{1}{2} \vec{u}^T \mat{M} \mat{D} \vec{u}
  - \frac{1}{2} \vec{u}^T \mat{R}[^T] \mat{B} \mat{R} \vec{u}
  + \frac{1}{2} \vec{u}^T \mat{D}[^T] \mat{M} \vec{u}
  - \vec{u}^T \mat{R}[^T] \mat{B} \left( \vecfnum - \mat{R} \vec{u} \right)
  \\
  &=
  \frac{1}{2} \vec{u}^T \mat{R}[^T] \mat{B} \mat{R} \vec{u}
  - \vec{u}^T \mat{R}[^T] \mat{B} \vecfnum,
\end{aligned}
\end{equation}
where the SBP property \eqref{eq:SBP} has been used in the second line.
Writing the element index as an upper index and suppressing the index $\cdot^{(i)}$
for the $i$-th element, this can be rewritten as
\begin{equation}
\label{eq:const-lin-adv-entropy-target}
\begin{aligned}
  \vec{u}^T \mat{M} \partial_t \vec{u}
  &=
  \frac{1}{2} \vec{u}^T \mat{R}[^T] \mat{B} \mat{R} \vec{u}
  - \vec{u}^T \mat{R}[^T] \mat{B} \vecfnum
  \\
  &=
  \frac{1}{2} \left( u_R^2 - u_L^2 \right)
  - \left(
    u_R \fnum\bigl(u_R, u_L^{(i+1)}\bigr)
    - u_L \fnum\bigl(u_R^{(i-1)}, u_L\bigr)
  \right)
  \\
  &=
  \left(
    \frac{u_R^{(i-1)} + u_L}{2} \fnum\bigl(u_R^{(i-1)}, u_L\bigr)
    - \frac{u_L^2}{2}
  \right)
  - \left(
    \frac{(u_R + u_L^{(i+1)}}{2} \fnum\bigl(u_R, u_L^{(i+1)}\bigr)
    - \frac{u_R^2}{2}
  \right)
  \\
  &=
  - \vec{1}^T \mat{R}^T \mat{B} \vecFnum,
  \quad
  \vecFnum = \left(
    \Fnum\bigl(u_R^{(i-1)}, u_L\bigr), \Fnum\bigl(u_R, u_L^{(i+1)}\bigr)
  \right)^T,
\end{aligned}
\end{equation}
where $\Fnum(u_-, u_+) = (u_- + u_+)/2 \cdot \fnum(u_-,u_+) - (\psi(u_-) + \psi(u_+))/2$
is the entropy flux of \citet{tadmor1987numerical}. The basic idea of
\citet{abgrall2018general} is to enforce \eqref{eq:const-lin-adv-entropy-target}
for any semidiscretisation via the addition of a correction term $\vec{r}$ on the
left hand side of \eqref{eq:const-lin-adv-DG-standard} that is consistent with
zero and does not violate the conservation relation (using $\mat{D} \vec{1} = \vec{0}$)
\begin{equation}
\label{eq:const-lin-adv-conservation}
\begin{aligned}
  \vec{1}^T \mat{M} \partial_t \vec{u}
  &=
  - \vec{1}^T \mat{M} \mat{D} \vec{u}
  - \vec{1}^T \mat{R}[^T] \mat{B} \left( \vecfnum - \mat{R} \vec{u} \right)
  =
  - \vec{1}^T \mat{R}^T \mat{B} \vecfnum.
\end{aligned}
\end{equation}
He proposes a correction term of the form
\begin{equation}
\label{eq:const-lin-adv-correction}
\begin{gathered}
  \vec{r} = \alpha \left(
    \vec{u}
    - \frac{\vec{1}^T \mat{M} \vec{u}}{\vec{1}^T \mat{M} \vec{1}} \vec{1}
  \right),
  \quad
  \alpha = \frac{\mathcal{E}}{\vec{u}^T \mat{M} \vec{u}
                        - \frac{(\vec{1}^T \mat{M} \vec{u})^2}{\vec{1}^T \mat{M} \vec{1}}},
  \\
  \mathcal{E} = \vec{1}^T \mat{R}^T \mat{B} \vecFnum - \vec{w}^T \mat{M} \mat{D} \vec{u}
                  - \vec{u}^T \mat{R}[^T] \mat{B} \left( \vecfnum - \mat{R} \vec{u} \right).
\end{gathered}
\end{equation}
Indeed, the conservation relation \eqref{eq:const-lin-adv-conservation} is left
unchanged, since
\begin{equation}
  \vec{1}^T \mat{M} \vec{r}
  =
  \alpha \left(
    \vec{1}^T \mat{M} \vec{u}
    - \frac{\vec{1}^T \mat{M} \vec{u}}{\vec{1}^T \mat{M} \vec{1}} \vec{1}^T \mat{M} \vec{1}
  \right)
  =
  \vec{0}.
\end{equation}
Moreover, the entropy rate satisfies the desired equation
\eqref{eq:const-lin-adv-entropy-target}, since
\begin{equation}
\begin{aligned}
  \vec{u}^T \mat{M} \partial_t \vec{u}
  &=
  - \vec{u}^T \mat{M} \vec{r}
  - \vec{u}^T \mat{M} \mat{D} \mat{u}
  - \vec{u}^T \mat{R}[^T] \mat{B} \left( \vecfnum - \mat{R} \vec{u} \right)
  =
  \\
  &=
  - \underbrace{\alpha \left(
    \vec{u}^T \mat{M} \vec{u}
    - \frac{\vec{1}^T \mat{M} \vec{u}}{\vec{1}^T \mat{M} \vec{1}} \vec{u}^T \mat{M} \vec{1}
  \right)}_{= \mathcal{E}}
  - \vec{u}^T \mat{M} \mat{D} \mat{u}
  - \vec{u}^T \mat{R}[^T] \mat{B} \left( \vecfnum - \mat{R} \vec{u} \right)
  \\
  &=
  - \vec{1}^T \mat{R}^T \mat{B} \vecFnum.
\end{aligned}
\end{equation}
If the denominator of $\alpha$ in \eqref{eq:const-lin-adv-correction} is zero,
the numerical solution is constant in the element because of the Cauchy Schwarz
inequality (since $\vec{1}$ and $\vec{u}$ are linearly dependent in that case). Then,
the DG scheme reduces to a finite volume scheme using the numerical flux $\fnum$
and is therefore entropy conservative/stable depending on $\fnum$.

In the following, some numerical experiments will be conducted using nodal DG
methods on equidistant nodes including the boundaries and diagonal mass matrices
using the weights of the closed Newton Cotes quadrature formula (which are positive
in the cases considered below). The domain is divided into $N = 16$ elements and
polynomials of degree $\leq p = 3$ are applied. The initial condition is advanced
one time step $\Delta t$ using SSPRK(10,4) of \cite{ketcheson2008highly}.

These methods have been implemented in Julia \texttt{v0.6.4} \citep{bezanson2017julia}
using floating point numbers with extended precision (\texttt{BigFloat} with
\texttt{setprecision(5000)}). Using $500$ different values of $\Delta t$ (with
uniformly distributed logarithms), the discrete energy errors after one time step
$\norm{u_\mathrm{num}(\Delta t)}_M^2 - \norm{u_0}_M^2$ (computed via the mass matrix
$\mat{M}$) are shown in Figure~\ref{fig:entropy_stability_via_Abgrall}. As can be
seen there, the discrete energy increases for every choice of $\Delta t > 0$.
Moreover, the increase of the energy scales as $\O(\Delta t^5)$, as expected for
one time step of a fourth order Runge--Kutta method.

To sum up, a semidiscrete DG scheme with insufficient quadrature strength
is made semidiscretely energy conservative following the approach of
\citet{abgrall2018general}. Integrating the resulting ordinary differential equation
in time with SSPRK(10,4) results in monotonically increasing energies, even for
ridiculously small time steps $\dt > 0$.
\begin{figure}[!ht]
\centering
\captionsetup{aboveskip=-4pt}
  \includegraphics{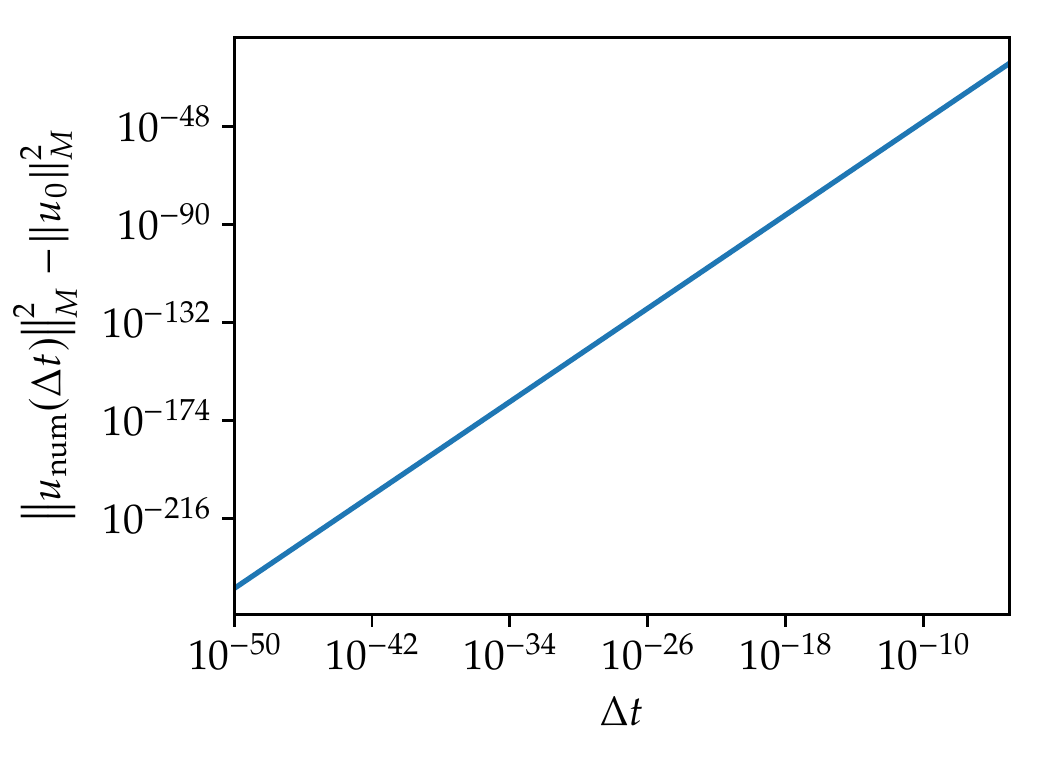}
  \caption{Discrete energy errors after one time step of SSPRK(10,4).}
  \label{fig:entropy_stability_via_Abgrall}
\end{figure}

\section{First Order Schemes}
\label{sec:first-order}

In contrast to the negative results of the previous sections for explicit methods
of at least second order of accuracy, there are first order schemes that are strongly
stable. To prove this, it suffices to consider schemes with two stages, i.e.
\begin{equation}
\label{eq:erk21}
\begin{aligned}
  u_1 &= u_0,
  \\
  u_2 &= u_0 + a_{21} \dt g(u_1),
  \\
  u_+ &= u_0 + b_1 \dt g(u_1) + b_2 \dt g(u_2).
\end{aligned}
\end{equation}
\begin{theorem}
\label{thm:first-order}
  There are first order accurate explicit Runge--Kutta method with two stages
  \eqref{eq:erk21} that are
  \begin{itemize}
    \item
    strong stability preserving

    \item
    and strongly stable for the ODE \eqref{eq:ode} with semibounded and Lipschitz
    continuous $g$ with $\normLip{g} \leq L$ under a time step constraint
    $0 < \dt \leq \dt_\mathrm{max} \propto L^{-1}$.
  \end{itemize}
\end{theorem}

\begin{proof}
  Inserting \eqref{eq:erk21} and using $u_1 = u_0$ yields
  \begin{equation}
  \begin{aligned}
    \norm{u_+}^2 - \norm{u_0}^2
    &=
    2 \dt \scp{u_0}{ b_1 g(u_1) + b_2 g(u_2)}
    \\&\quad
    + \dt^2 \Bigl(
      b_1^2 \norm{g(u_1)}^2 + 2 b_1 b_2 \scp{g(u_1)}{g(u_2)} + b_2^2 \norm{g(u_2)}^2
    \Bigr)
    \\
    &=
    2 b_1 \dt \scp{u_0}{g(u_0)} + 2 b_2 \dt \scp{u_2 - a_{21} \dt g(u_0)}{g(u_2)}
    \\&\quad
    + \dt^2 \Bigl(
      b_1^2 \norm{g(u_0)}^2 + 2 b_1 b_2 \scp{g(u_0)}{g(u_2)} + b_2^2 \norm{g(u_2)}^2
    \Bigr)
    \\
    &=
    2 b_1 \dt \scp{u_0}{g(u_0)} + 2 b_2 \dt \scp{u_2}{g(u_2)}
    \\&\quad
    + \dt^2 \Bigl(
      b_1^2 \norm{g(u_0)}^2 + 2 (b_1 b_2 - b_2 a_{21}) \scp{g(u_0)}{g(u_2)} + b_2^2 \norm{g(u_2)}^2
    \Bigr).
  \end{aligned}
  \end{equation}
  Since $g$ is semibounded, the inner products can be estimated as
  $\scp{u_i}{g(u_i)} \leq 0$. Thus, the terms proportional to $\dt$ can be estimated
  if $b_1,b_2 \geq 0$. Since the conditions for first order are $b_1 + b_2 = 1$,
  this becomes
  \begin{equation}
  \label{eq:erk21-b-geq-0}
    0 \leq b_2 \leq 1,
    \quad
    b_1 = 1 - b_2.
  \end{equation}

  The terms multiplied by $\dt^2$ satisfy
  \begin{equation}
  \begin{aligned}
    &\quad
    b_1^2 \norm{g(u_0)}^2
    + 2 (b_1 b_2 - b_2 a_{21}) \scp{g(u_0)}{g(u_2)}
    + b_2^2 \norm{g(u_2)}^2
    \\
    &=
    b_1^2 \norm{g(u_0)}^2
    + 2 (b_1 b_2 - b_2 a_{21}) \scp{g(u_0)}{g(u_0) + g(u_2) - g(u_0)}
    + b_2^2 \norm{g(u_0) + g(u_2) - g(u_0)}^2
    \\
    &=
    \bigl( b_1^2 + 2 (b_1 b_2 - b_2 a_{21}) + b_2^2 \bigr) \norm{g(u_0)}^2
    \\&\quad
    + 2 (b_1 b_2 - b_2 a_{21} + b_2^2) \scp{g(u_0)}{g(u_2) - g(u_0)}
    + b_2^2 \norm{g(u_2) - g(u_0)}^2
    \\
    &\leq
    \bigl( (b_1 + b_2)^2 - 2 b_2 a_{21} \bigr) \norm{g(u_0)}^2
    \\&\quad
    + 2 \abs{b_1 b_2 - b_2 a_{21} + b_2^2} \, \norm{g(u_0)} \, \norm{g(u_2) - g(u_0)}
    + b_2^2 \norm{g(u_2) - g(u_0)}^2
    \\
    &\leq
    \bigl( (b_1 + b_2)^2 - 2 b_2 a_{21} \bigr) \norm{g(u_0)}^2
    \\&\quad
    + 2 \abs{b_1 b_2 - b_2 a_{21} + b_2^2} \abs{a_{21}} L \dt \norm{g(u_0)}^2
    + b_2^2 a_{21}^2 L^2 \dt^2 \norm{g(u_0)}^2.
  \end{aligned}
  \end{equation}
  Inserting the order condition \eqref{eq:erk21-b-geq-0}, the last expression
  can be written as
  \begin{equation}
    \dots
    \leq
    (1 - 2 b_2 a_{21}) \norm{g(u_0)}^2
    + 2 \abs{1 - a_{21}} \, \abs{b_2 a_{21}} L \dt \norm{g(u_0)}^2
    + \abs{b_2 a_{21}}^2 L^2 \dt^2 \norm{g(u_0)}^2.
  \end{equation}
  If $g(u_0) = 0$, then $u_+ = u_2 = u_1 = u_0$ and strong stability is obvious.
  Otherwise, the term without $\dt$ is negative if
  \begin{equation}
  \label{eq:erk21-condition-strong-stability}
    1 - 2 b_2 a_{21} < 0.
  \end{equation}
  In that case, strong stability is achieved for sufficiently small $\dt$
  (and the natural assumption $b_2 a_{21} L \neq 0$), since
  \begin{equation}
    (1 - 2 b_2 a_{21})
    + 2 \abs{1 - a_{21}} \, \abs{b_2 a_{21}} L \dt
    + \abs{b_2 a_{21}}^2 L^2 \dt^2
    \leq
    0
  \end{equation}
  for
  \begin{equation}
  \label{eq:erk21-strong-stability-CFL}
    \dt
    \leq
    \frac{\sqrt{ (1 - a_{21})^2 - (1 - 2 b_2 a_{21}) } - \abs{1 - a_{21}}}{\abs{b_2 a_{21}} L}.
  \end{equation}
  Thus, there are strongly stable schemes.

  Choosing for example
  \begin{equation}
    b_1 = b_2 = \frac{1}{2},
    \quad
    a_{21} = \frac{3}{2},
  \end{equation}
  the new value $u_+$ can be written as
  \begin{equation}
  \begin{aligned}
    u_+
    &=
    u_0 + \frac{1}{2} \dt g(u_0) + \frac{1}{2} \dt g(u_2)
    \\
    &=
    \frac{3}{4} u_0 + \frac{1}{2} \dt g(u_0)
    + \frac{1}{4} \left( u_2 - \frac{3}{2} \dt g(u_0)  \right) + \frac{1}{2} \dt g(u_2)
    \\
    &=
    \frac{3}{4} \left( u_0 + \frac{1}{6} \dt g(u_0) \right)
    + \frac{1}{4} \left( u_2 + 2 \dt g(u_2) \right).
  \end{aligned}
  \end{equation}
  Since $u_2 = u_0 + \frac{3}{2} \dt g(u_0)$, $u_+$ is a convex combination of
  explicit Euler steps with positive step sizes, the resulting scheme is strong
  stability preserving.
\end{proof}

\begin{remark}
  Of course, the scheme constructed in the proof of Theorem~\ref{thm:first-order}
  and the derived lower bound on the SSP coefficient are not optimal. However,
  since such first order schemes are not really relevant in practice, no attempt
  to optimise them has been made.
\end{remark}

\begin{remark}
  Generalising the approach used in proof of Theorem~\ref{thm:first-order}, it
  can be expected that there are strongly stable schemes of first order if more
  stages are used.
\end{remark}

\begin{remark}
  If a non-autonomous problem \eqref{eq:ode-f} is considered, the proof of
  Theorem~\ref{thm:first-order} requires Lipschitz continuity in $(t,u)$ instead
  of continuity in $t$ and Lipschitz continuity in $u$ as required for the
  Picard-Lindelöf theorem, since $f(t_2, u_2) - f(t_0, u_0)$ has to be estimated.
\end{remark}

\begin{remark}
  Since higher-order schemes satisfy $\norm{u_+}^2 - \norm{u_0}^2 = \O(\dt^p)$
  for $p > 2$ if $g$ is smooth, it does not seem to be (easily) possible to get
  similar estimate for higher order schemes using only Lipschitz continuity of
  $g$.
\end{remark}

\section{Summary and Discussion}
\label{sec:summary}

In this article, strong stability of explicit SSP Runge--Kutta methods has been
investigated. Many well-known and widespread high order schemes are not strongly
stable for ODEs with general nonlinear, smooth, and semibounded operators with
bounded Lipschitz constant, cf. Theorems~\ref{thm:ERK-SSP-3-2}, \ref{thm:ssprks2},
\ref{thm:ssprkn23}, and~\ref{thm:ssprk104}.
Moreover, it has been proven that the norms of the numerical solutions can even
increase monotonically and without bounds for the popular three stage, third order
method SSPRK(3,3) of \cite{shu1988efficient}, cf. Theorem~\ref{thm:ssprk33}.
Additionally, it has been shown in section~\ref{sec:ssprk104} that the ten stage,
fourth order method SSPRK(10,4) of \cite{ketcheson2008highly} can result in
increasing norms of the solution for an energy stable and nonlinear
semidiscretisation of a hyperbolic conservation law. Finally, it has been proven
that such restrictions do not apply to first order Runge--Kutta methods, cf.
Theorem~\ref{thm:first-order}. In particular, there are strongly stable SSP methods,
even for nonlinear and semibounded operators that are Lipschitz continuous. In
that case, strong stability can be guaranteed under a time step restriction
$\dt \leq \dt_\mathrm{max}$, where $\dt_\mathrm{max}$ is proportional to the
inverse of the Lipschitz constant of the right hand side.

It is well-known that implicit Runge--Kutta methods can have more favourable stability
properties than explicit ones. In particular, there are strongly stable methods
for general semibounded operators \citep[sections~357--359]{butcher2016numerical}.
Furthermore, summation by parts operators can be used to construct schemes with these
properties \citep{nordstrom2013summation, lundquist2014sbp, boom2015high}, resulting
e.g. in energy stable schemes for nonlinear equations \citep{nordstrom2018energy}.
This is also related to space-time discontinuous Galerkin schemes, where entropy
stability can be obtained \citep{friedrich2018entropy}.

In this light, it seems interesting to investigate whether there are general
possibilities to obtain strong stability of explicit Runge--Kutta methods by
approximating the original problem, e.g. by adding sufficient artificial dissipation,
cf. \cite{zakerzadeh2016high}. In the light of the current results, it might be
conjectured that such a dissipative mechanism might be necessary to obtain strong
stability with explicit methods. If this is possible, it is interesting to compare
such schemes with fully implicit ones.

\appendix

\section*{Acknowledgements}

The author was supported by the German Research Foundation (DFG, Deutsche
Forschungsgemeinschaft) under Grant SO~363/14-1. The author would like to thank
David Ketcheson very much for some comments on an earlier draft of this manuscript
and for pointing out the references \cite{dahlquist1979generalized, dekker1984stability}.

\printbibliography

\end{document}